\documentclass{amsart}

\usepackage{amsmath} 
\usepackage{amssymb}

\newcount\skewfactor
\def\mathunderaccent#1#2 {\let\theaccent#1\skewfactor#2
\mathpalette\putaccentunder}
\def\putaccentunder#1#2{\oalign{$#1#2$\crcr\hidewidth
\vbox to.2ex{\hbox{$#1\skew\skewfactor\theaccent{}$}\vss}\hidewidth}}
\def\name{\mathunderaccent\tilde-3 }

\newcommand{\forces}{\Vdash} 
\newcommand{\con}{{\mathfrak c}}

\newcommand{\can}{{}^{\omega}2}

\newcommand{\rest}{{\restriction}}
 
\newcommand{\rng}{{\rm rng}}
\newcommand{\conc}{{}^\frown\!}
\newcommand{\vtl}{\vartriangleleft} 
\newcommand{\vare}{\varepsilon}

\newcommand{\cf}{{\rm cf}}
\newcommand{\rk}{{\rm rk}}

\newcommand{\stnd}{{\rm stnd}}
\newcommand{\ndrk}{{\rm ndrk}}
\newcommand{\NDRK}{{\rm NDRK}}
\newcommand{\Mtk}{{{\mathbf M}_{\bar{T},k}}}
\newcommand{\fMtk}{{{\mathbf M}^n_{\bar{t},k}}}

\newcommand{\cA}{{\mathcal A}}

\newcommand{\cB}{{\mathcal B}}

\newcommand{\bbC}{{\mathbb C}}

\newcommand{\cL}{{\mathcal L}}

\newcommand{\cN}{{\mathcal N}}
\newcommand{\cM}{{\mathcal M}}

\newcommand{\bbM}{{\mathbb M}}
\newcommand{\bm}{{\mathbf m}}
\newcommand{\bn}{{\mathbf n}}

\newcommand{\bbP}{{\mathbb P}}
\newcommand{\bbQ}{{\mathbb Q}}

\newcommand{\bV}{{\mathbf V}}

\newcommand{\bbZ}{{\mathbb Z}}

\newtheorem{theorem}{Theorem}[section] 
\newtheorem{claim}{Claim}[theorem]
\newtheorem{lemma}[theorem]{Lemma} 
\newtheorem{proposition}[theorem]{Proposition} 
\newtheorem{corollary}[theorem]{Corollary} 
\newtheorem{observation}[theorem]{Observation} 

\theoremstyle{definition}
\newtheorem{problem}[theorem]{Problem} 
 
\newtheorem{definition}[theorem]{Definition}
\newtheorem{hypothesis}[theorem]{Hypothesis}

\theoremstyle{remark}

\newtheorem{remark}[theorem]{Remark}

\begin{document}

\title{Borel sets without perfectly many overlapping translations}

\author{Andrzej Ros{\l}anowski}
\address{Department of Mathematics\\
University of Nebraska at Omaha\\
Omaha, NE 68182-0243, USA}
\email{aroslanowski@unomaha.edu}

\author{Saharon Shelah}
\address{Institute of Mathematics\\
 The Hebrew University of Jerusalem\\
 91904 Jerusalem, Israel\\
 and  Department of Mathematics\\
 Rutgers University\\
 New Brunswick, NJ 08854, USA}
\email{shelah@math.huji.ac.il}
\urladdr{http://shelah.logic.at}

\thanks{Publication 1138 of the second author.}

\date{June, 2018}

\begin{abstract}
For a cardinal $\lambda<\lambda_{\omega_1}$ we give a ccc forcing notion $\bbP$
such that 
\[\begin{array}{l}
\forces_{\bbP}\mbox{`` some $\Sigma^0_2$   set  $B\subseteq\can$ admits a
    sequence  }\langle\eta_\alpha:\alpha<\lambda\rangle\mbox{ of distinct
    elements of }\can\\
\qquad \mbox{such that }\ \big|(\eta_\alpha+B)\cap (\eta_\beta+B)\big|\geq
    6\mbox{ for all  }\alpha,\beta<\lambda\\ 
\qquad \mbox{but does not have a perfect set of such $\eta$'s ''.}
\end{array}\]
The construction closely follows the one from Shelah \cite[Section
1]{Sh:522}. 
\end{abstract}

\maketitle 

\section{Introduction}
Shelah \cite{Sh:522} analyzed when there are Borel in the plane which
contain large squares but no perfect squares. A rank on models with a
countable vocabulary was introduced and a used to define a cardinal
$\lambda_{\omega_1}$ (the first $\lambda$ such that there is no model with
universe $\lambda$, countable vocabulary and rank $<\omega_1$).
It was shown in \cite[Claim 1.12]{Sh:522} that every Borel set $B\subseteq
\can\times\can$ which contains a $\lambda_{\omega_1}$--square must contain a
perfect square. On the other hand, by \cite[Theorem 1.13]{Sh:522}, if
$\mu=\mu^{\aleph_0}< \lambda_{\omega_1}$ then some ccc forcing notion forces 
that (the continuum is arbitrarily large and) some Borel set contains a
$\mu$--square but no $\mu^+$--square. 

We would like to understand what the results mentioned above mean for
general relations. Natural first step is to ask about Borel sets with
$\mu\geq \aleph_1$ pairwise disjoint translations but without any perfect
set of such translations, as motivated e.g.~by Balcerzak, Ros{\l}anowski and 
Shelah \cite{BRSh:512} (were we studied the $\sigma$--ideal of subsets of
$\can$ generated by Borel sets with a perfect set of pairwise disjoint
translations). A generalization of this direction could follow Zakrzewski
\cite{Zak13} who introduced perfectly $k$--small sets. 

However, preliminary analysis of the problem revealed that another, somewhat
orthogonal to the one described above, direction is more natural in the
setting of \cite{Sh:522}.  Thus we investigate Borel sets with many, but not
too many, pairwise overlapping intersections. 

Easily, every uncountable Borel subset $B$ of $\can$ has a perfect set of
pairwise non-disjoint translations (just consider a perfect set $P\subseteq
B$ and note that for $x,y\in P$ we have ${\mathbf 0},x+y\in (B+x)\cap
(B+y)$). The problem of many non-disjoint translations becomes more
interesting if we demand that the intersections have more elements. Note
that in $\can$, if $x+b_0=y+b_1$ then also $x+b_1=y+b_0$, so $x\neq y$ and
$|(B+x)\cap (B+y)|<\omega$ imply that $|(B+x)\cap (B+y)|$ is even. 

In the present paper we study the case when the intersections $(B+x)\cap
(B+y)$ have at least 6 elements. We show that for
$\lambda<\lambda_{\omega_1}$ there is a ccc forcing notion $\bbP$ adding a
$\Sigma^0_2$ subset $B$ of the Cantor space $\can$ such that    
\begin{itemize}
\item for some $H\subseteq \can$ of size $\lambda$, $|(B+h)\cap (B+h')|\geq
  6$ for all $h,h'\in H$, but 
\item for every perfect set $P\subseteq \can$ there are $x,x'\in P$ with 
$|(B+x)\cap (B+x')|<6$.
\end{itemize}

We fully utilize the algebraic properties of $(\can,+)$, in particular the
fact that all elements of $\can$ are self-inverse. The general case of
Polish groups will be investigated in the subsequent work
\cite{Sh:F1771}. 

In Section 2 of the paper we recall the rank from \cite{Sh:522}. We give the
relevant definitions, state and prove all the properties needed for our
results later. In the third section we analyze when a $\Sigma^0_2$ subset of
$\can$ has a perfect set of pairwise overlapping translations. The main
consistency result concerning adding a Borel set with no perfect set of
overlapping translations is given in the fourth section.
\medskip

\noindent{\bf Notation}:\qquad Our notation is rather standard and
compatible with that of classical textbooks (like Jech \cite{J} or
Bartoszy\'nski and Judah \cite{BaJu95}). However, in forcing we keep the
older convention that {\em a stronger condition is the larger one}.

\begin{enumerate}
\item For a set $u$ we let 
\[u^{\langle 2\rangle}=\{(x,y)\in u\times u:x\neq y\}.\]\
\item The Cantor space $\can$ of all infinite sequences with values 0 and 1
is equipped with the natural product topology and the group operation of
coordinate-wise addition $+$ modulo 2.   
\item Ordinal numbers will be denoted be the lower case initial letters of
  the Greek alphabet $\alpha,\beta,\gamma,\delta$. Finite ordinals
  (non-negative integers) will be denoted by letters
  $a,b,c,d,i,j,k,\ell,m,n,M$ and $\iota$. 
\item The Greek letters $\kappa,\lambda$ will stand for uncountable
  cardinals. 
\item For a forcing notion $\bbP$, all $\bbP$--names for objects in
  the extension via $\bbP$ will be denoted with a tilde below (e.g.,
  $\name{\tau}$, $\name{X}$), and $\name{G}_\bbP$ will stand for the
  canonical $\bbP$--name for the generic filter in $\bbP$.
\end{enumerate}

\section{The rank}
We will remind some basic facts from \cite[Section 1]{Sh:522} concerning a
rank (on models with countable vocabulary) which will be used in the
construction of a forcing notion in the next section. For the convenience of
the reader we provide proofs for most of the claims, even though they were
given in \cite{Sh:522}. Our rank $\rk$ is the $\rk^0$ of \cite{Sh:522} and
$\rk^*$ is the $\rk^2$ there.  
\bigskip

Let $\lambda$ be a cardinal and $\bbM$ be a model with the universe $\lambda$
and a countable vocabulary $\tau$.

\begin{definition}
\label{defofrank}
\begin{enumerate}
\item By induction on ordinals $\alpha$, for 
finite non-empty sets $w\subseteq\lambda$ we define when $\rk(w,\bbM)\geq  
\alpha$. Let $w=\{a_0,\ldots,a_n\} \subseteq\lambda$, $|w|=n+1$. 
 \begin{enumerate}
\item[(a)] $\rk(w)\geq 0$ if and only if for every quantifier free formula 
    $\varphi\in \cL(\tau)$ and each $k\leq n$, if 
$\bbM\models \varphi[a_0,\ldots,a_k,\ldots,a_n]$ then  the set 
\[\big\{a\in \lambda:\bbM\models \varphi[a_0,\ldots,a_{k-1},a,a_{k+1},
\ldots,a_n]\big\}\]  is 
uncountable;  
\item[(b)] if $\alpha$ is limit, then $\rk(w,\bbM)\geq\alpha$ if and only if 
  $\rk(w,\bbM)\geq\beta$ for all $\beta<\alpha$;  
\item[(c)] $\rk(w,\bbM)\geq\alpha+1$ if and only if for every quantifier free 
  formula $\varphi\in \cL(\tau)$ and each $k\leq n$, if $\bbM\models 
  \varphi[a_0,\ldots,a_k,\ldots,a_n]$ then there is $a^*\in\lambda\setminus 
  w$ such that 
\[\rk(w\cup\{a^*\},\bbM)\geq \alpha\quad\mbox{ and }\quad \bbM\models 
  \varphi[a_0,\ldots,a_{k-1},a^*,a_{k+1},\ldots,a_n].\]
 \end{enumerate}
\item Similarly, for finite non-empty sets $w\subseteq\lambda$ we define
  when $\rk^*(w,\bbM)\geq \alpha$ (by induction on ordinals $\alpha$). Let
  $w=\{a_0,\ldots,a_n\} \subseteq\lambda$. We take clauses (a) and (b) above
  and 
 \begin{enumerate}
\item[(c)$^*$] $\rk^*(w,\bbM)\geq\alpha+1$ if and only if for every quantifier free 
  formula $\varphi\in \cL(\tau)$ and each $k\leq n$, if $\bbM\models 
  \varphi[a_0,\ldots,a_k,\ldots,a_n]$ then there are pairwise distinct
  $\langle a^*_i:i<\omega_1\rangle \subseteq \lambda\setminus
  (w\setminus\{a_k\})$ such that $a_0^*=a_k$ and for all $i<j<\omega_1$ we
  have  
\[\rk^*(w\setminus\{a_k\}\cup\{a^*_i,a^*_j\},\bbM)\geq \alpha\quad\mbox{ and  
}\quad \bbM\models \varphi[a_0,\ldots,a_{k-1},a^*_i,a_{k+1},\ldots,a_n].\] 
\end{enumerate}
\end{enumerate}
\end{definition}

By a straightforward induction on $\alpha$ one easily shows the following
observation.  

\begin{observation}
\label{monot}
If $\emptyset\neq v\subseteq w$ then 
\begin{itemize}
\item  $\rk(w,\bbM)\geq\alpha\geq\beta$ implies $\rk(v,\bbM)\geq \beta$, and 
\item  $\rk^*(w,\bbM)\geq\alpha\geq\beta$ implies $\rk^*(v,\bbM)\geq \beta$.
\end{itemize}
\end{observation}

Hence we may define the rank functions on finite non-empty subsets of
$\lambda$. 

\begin{definition}
The ranks $\rk(w,\bbM)$ and $\rk^*(w,\bbM)$ of a finite non-empty set
$w\subseteq\lambda$ are defined as: 
 \begin{itemize}
\item $\rk(w,\bbM)=-1$ if $\neg (\rk(w,\bbM)\geq 0)$, and\\
$\rk^*(w,\bbM)=-1$ if $\neg (\rk^*(w,\bbM)\geq 0)$, 
\item $\rk(w,\bbM)=\infty$ if $\rk(w,\bbM)\geq \alpha$ for all ordinals $\alpha$,
  and\\
$\rk^*(w,\bbM)=\infty$ if $\rk^*(w,\bbM)\geq \alpha$ for all ordinals $\alpha$,
\item for an ordinal $\alpha$: $\rk(w,\bbM)=\alpha$ if $\rk(w,\bbM)\geq \alpha$
  but $\neg(\rk(w,\bbM)\geq\alpha+1)$,\\ 
and $\rk^*(w,\bbM)=\alpha$ if $\rk^*(w,\bbM)\geq \alpha$ but
  $\neg(\rk^*(w,\bbM)\geq\alpha+1)$.
  \end{itemize}
\end{definition}

\begin{definition}
  \begin{enumerate}
  \item For an ordinal $\vare$ and a cardinal $\lambda$ let ${\rm 
  NPr}_\vare(\lambda)$  be the following statement: ``there is a model 
$\bbM^*$ with the universe $\lambda$ and a countable vocabulary $\tau^*$ such 
that  $\sup\{\rk(w,\bbM^*):\emptyset\neq w\in [\lambda]^{<\omega}
\}<\vare$.''    
\item The statement ${\rm NPr}^*_\vare(\lambda)$ is defined similarly but
  using the rank $\rk^*$. 
\item ${\rm Pr}_\vare(\lambda)$ and ${\rm Pr}^*_\vare(\lambda)$ are the
  negations of ${\rm NPr}_\vare(\lambda)$ and ${\rm NPr}^*_\vare(\lambda)$,
  respectively.  
  \end{enumerate}

\end{definition}

\begin{observation}
\label{obsonrk}
  \begin{enumerate}
\item If a model $\bbM^+$ (on $\lambda$) is an expansion of the model $\bbM$, then
  $\rk^*(w,\bbM^+)\leq \rk(w,\bbM^+)\leq \rk(w,\bbM)$.  
\item If $\lambda$ is uncountable and ${\rm NPr}_\vare(\lambda)$, then
  there is a model $\bbM^*$ with the universe $\lambda$ and a countable
  vocabulary $\tau^*$ such that 
\begin{itemize}
\item $\rk(\{a\},\bbM^*)\geq 0$ for all $a\in \lambda$ and 
\item  $\rk(w,\bbM^*)<\vare$ for every finite non-empty set $w\subseteq \lambda$.
\end{itemize}
  \end{enumerate}
\end{observation}

\begin{proposition}
[See {\cite[Claim 1.7]{Sh:522}}]
\label{cl1.7-522}
\begin{enumerate}
\item ${\rm NPr}_1(\omega_1)$.
\item If ${\rm NPr}_\vare(\lambda)$, then ${\rm
    NPr}_{\vare+1}(\lambda^+)$. 
\item If ${\rm NPr}_\vare(\mu)$ for $\mu<\lambda$ and
  $\cf(\lambda)=\omega$, then ${\rm NPr}_{\vare+1}(\lambda)$. 
\item ${\rm NPr}_\vare(\lambda)$ implies ${\rm NPr}^*_\vare(\lambda)$.
\end{enumerate}
\end{proposition}

\begin{proof}
(1)\quad Let $Q$ be a binary relational symbol and let $\bbM_1$ be a model with
the universe $\omega_1$, the vocabulary $\tau(\bbM_1)=\{Q\}$ and such that 
$Q^{\bbM_1}=\{ (\alpha,\beta)\in \omega_1\times\omega_1:
\alpha<\beta\}$. Then for each $\alpha_0<\alpha_1<\omega_1$ we have
$\bbM_1\models Q[\alpha_0,\alpha_1]$ but the set $\{\alpha<\omega_1:
\bbM_1\models Q[\alpha,\alpha_1]\}$ is countable. Hence $\rk(w,\bbM_1)=-1$
whenever $|w|\geq 2$ and $\rk(\{\alpha\},\bbM_1)=0$ for $\alpha\in
\omega_1$. Consequently, $\bbM_1$ witnesses ${\rm   NPr}_1(\omega_1)$.  
\medskip

\noindent (2)\quad Assume ${\rm NPr}_\vare(\lambda)$ holds true as witnessed
by a model $\bbM$ with the universe $\lambda$ and a countable vocabulary
$\tau$. We may assume that $\tau=\{R_i:i<\omega\}$, where each $R_i$ is a
relational symbol of arity $n(i)$. Let $S$ be a new binary relational
symbol, $T$ be a new unary relational symbol, and $Q_i$ be a new
$(n(i)+1)$--ary relational symbol (for $i<\omega$). Let
$\tau^+=\{R_i,Q_i: i<\omega\}\cup \{S,T\}$.

For each $\gamma\in [\lambda,\lambda^+)$ fix a bijection $f_\gamma:\gamma
\stackrel{\rm 1-1}{\longrightarrow} \lambda$ with $f_\lambda$ being the
identity. We define a model $\bbM^+$:
\begin{itemize}
\item the vocabulary of $\bbM^+$ is $\tau^+$ and the universe of $\bbM^+$ is
  $\lambda^+$,  
\item $R_i^{\bbM^+}=R_i^\bbM\subseteq \lambda^{n(i)}$,
\item $Q_i^{\bbM^+}=\{\big(a_0,\ldots,a_{n(i)-1},a_{n(i)}\big):
  \lambda\leq a_{n(i)}<\lambda^+\ \&\ (\forall \ell<n(i))(a_\ell<
  a_{n(i)})\ \&\ \big(f_{a_{n(i)}}(a_0),\ldots,f_{a_{n(i)}}(
  a_{n(i)-1})\big)\in R_i^\bbM\}$,
\item $S^{\bbM^+}=\{(a_0,a_1)\in\lambda^+\times\lambda^+:
  a_0<a_1\}$ and $T^{\bbM^+}=[\lambda,\lambda^+)$. 
\end{itemize}

\begin{claim}
\label{cl1}
\begin{enumerate}
\item[(i)] If $\lambda\leq\gamma<\lambda^+$, $\emptyset\neq 
  w\subseteq \gamma$, then $\rk(w\cup\{\gamma\},\bbM^+)\leq 
  \rk(f_\gamma[w],\bbM)$ and thus $\rk(w\cup\{\gamma\},\bbM^+) <\vare$. 
\item[(ii)] If $\emptyset\neq w\subseteq\lambda$, then $\rk(w,\bbM^+)\leq  
  \rk(w,\bbM)$ and thus $\rk(w,\bbM^+) <\vare$.     
\item[(iii)] If $\lambda\leq\gamma<\lambda^+$, then $\rk(\{\gamma\},\bbM^+)  
  \leq\vare$. 
\end{enumerate}
\end{claim}

\begin{proof}[Proof of the Claim]
(i)\quad By induction on $\alpha$ we show that $\alpha\leq\rk(w\cup
\{\gamma\}, \bbM^+)$ implies $\alpha\leq \rk(f_\gamma[w],\bbM)$ (for all
sets $w\subseteq \gamma$ with fixed $\gamma\in
[\lambda,\lambda^+)$).
\smallskip

\noindent $(*)_0$\quad Assume $\rk(w\cup\{\gamma\},\bbM^+)\geq 0$,
$w=\{a_0,\ldots,a_n\}$ and $k\leq n$. Let $\varphi(x_0,
\ldots,x_n)$ be a quantifier free formula in the vocabulary $\tau$
such that 
\[\bbM\models \varphi[f_\gamma(a_0), \ldots,
f_\gamma(a_k),\ldots,f_\gamma(a_n)].\]
Let $\varphi^*(x_0,\ldots, x_n, x_{n+1})$ be a quantifier free formula
in the vocabulary $\tau^+$ obtained from $\varphi$ by replacing each
$R_i(y_0,\ldots, y_{n(i)-1})$ (where $\{y_0,\ldots, y_{n(i)-1}\}
\subseteq \{x_0,\ldots,x_n\}$) with
$Q_i(y_0,\ldots,y_{n(i)-1},x_{n+1})$ and let $\varphi^+$ be  
\[\varphi^*(x_0,\ldots,x_n,x_{n+1})\ \wedge\ S(x_0,x_{n+1})\
\wedge\ldots \wedge S(x_n,x_{n+1}).\] 
Then  $\bbM^+\models\varphi^+[a_0,\ldots,a_k,\ldots,a_n,
\gamma]$. By our assumption on $w\cup\{\gamma\}$ we know that the set 
$A=\{b<\lambda^+:\bbM^+\models \varphi^+[a_0,\ldots,a_{k-1},
b, a_{k+1},\ldots,a_n,\gamma]\}$ is uncountable. Clearly
$A\subseteq\gamma$ (note $S(x_k,x_{n+1})$ in $\varphi^+$) and thus the
set $f_\gamma[A]$ is an uncountable subset of $\lambda$. For each
$b\in A$ we have $\bbM\models \varphi[f_\gamma(a_0),\ldots,
f_\gamma(b), \ldots,f_\gamma(a_n)]$, so now we may conclude
that $\rk(f_\gamma[w],\bbM)\geq 0$. 
\smallskip

\noindent $(*)_1$\quad Assume $\rk(w\cup\{\gamma\},\bbM^+)\geq \alpha+1$.
Let $\varphi(x_0,\ldots,x_n)$ be a quantifier free formula in the
vocabulary $\tau$, $k\leq n$ and $w=\{a_0,\ldots,a_n\}$, and suppose
that $\bbM\models\varphi[f_\gamma(a_0),\ldots,f_\gamma(a_k),\ldots,
f_\gamma(a_n)]$. Let $\varphi^*$ and $\varphi^+$ be defined exactly as in
$(*)_0$. Then $\bbM^+\models \varphi^+[a_0,\ldots,a_k,\ldots,a_n,\gamma]$. 
By our assumption there is $a^*\in \lambda^+\setminus (w\cup \{\gamma\})$
such that  $\bbM^+\models \varphi^+[a_0,\ldots,a^*,\ldots,a_n,\gamma]$ and
$\rk(w\cup\{\gamma,a^*\}, \bbM^+)\geq \alpha$. Necessarily $a^*<\gamma$, and by
the inductive hypothesis $\rk(f_\gamma[w\cup\{a^*\}],\bbM)\geq \alpha$. Clearly
$\bbM\models\varphi[f_\gamma(a_0),\ldots,f_\gamma(a^*),\ldots, f_\gamma(a_n)]$
and we may conclude $\rk(f_\gamma[w],\bbM)\geq \alpha+1$.  
\smallskip

\noindent $(*)_2$\quad If $\alpha$ is limit and
$\rk(w\cup\{\gamma\},\bbM^+)\geq \alpha$ then, by the inductive
hypothesis, for each $\beta<\alpha$ we have $\beta\leq\rk(w
\cup\{\gamma\}, \bbM^+)\leq \rk(f_\gamma[w],\bbM)$. Hence $\alpha\leq
\rk(f_\gamma[w],\bbM)$. 
\medskip

\noindent (ii)\quad Induction similar to part (i). For a quantifier free
formula $\varphi(x_0,\ldots,x_n)$  in the vocabulary $\tau$, let $\varphi^*$
be the formula $\varphi(x_0,\ldots,x_n)\ \wedge\ \neg T(x_0)\ \wedge \ldots
\wedge \neg T(x_n)$ (so $\varphi^*$ is a quantifier free formula in the
vocabulary $\tau^+$).   If $\varphi$ witnesses that $\neg(\rk(w,\bbM)\geq 0)$,
then $\varphi^*$ witnesses $\neg(\rk(w,\bbM^+)\geq 0)$, and similarly with 
$\alpha+1$ in place of 0. 
\medskip

\noindent (iii)\quad Suppose towards contradiction that
$\vare+1\leq\rk(\{\gamma\}, \bbM^+) $. Since $\bbM^+\models T[\gamma]$, we may
find $\gamma'\neq \gamma$ such that $\rk(\{\gamma,\gamma'\},\bbM^+)\geq \vare$ and
$\bbM^+\models T[\gamma']$. Let $\{\gamma,\gamma'\}=\{\gamma_0,\gamma_1\}$
where $\gamma_0<\gamma_1$. It follows from part (i) that $\rk(\{\gamma_0,  
\gamma_1\}, \bbM^+)<\vare$, a contradiction. 
\end{proof}

It follows from Claim \ref{cl1} (and Observation \ref{monot}) that
$\rk(w,\bbM^+)\leq \vare$ for every non-empty set $w\subseteq
\lambda^+$. Consequently, the model $\bbM^+$ witnesses  ${\rm
    NPr}_{\vare+1}(\lambda^+)$. 
\medskip

\noindent (3) Let $\langle \mu_n:n<\omega\rangle$ be an increasing sequence
cofinal in $\lambda$. For each $n$ fix a model $\bbM_n$ with a countable
vocabulary $\tau(\bbM_n)$ consisting of relational symbols only and with the
universe $\mu_n$ and such that $\rk(w,\bbM_n)<\vare$ for nonempty finite
$w\subseteq \mu_n$.  We also assume that $\tau(\bbM_n)\cap \tau(\bbM_m)=\emptyset$
for $n<m<\omega$. Let $P_n$ (for $n<\omega$) be new unary relational symbols
and let $\tau^+=\bigcup \{\tau(\bbM_n):
n<\omega\}\cup\{P_n:n<\omega\}$.
Consider a model $\bbM^+$ in vocabulary $\tau^+$ with the universe $\lambda$
and such that
\begin{itemize}
\item $P_n^\bbM=\mu_n$ for $n<\omega$, and 
\item for each $n<\omega$ and $S\in \tau(\bbM_n)$ we have $S^\bbM=S^{\bbM_n}$.
\end{itemize}

\begin{claim}
\label{cl2}
If $w$ is a finite non-empty subset of $\mu_n$, $n<\omega$, then
$\rk(w,\bbM)\leq \rk(w,\bbM_n)<\vare$.
\end{claim}

\begin{proof}[Proof of the Claim]
Similar to the proofs in Claim \ref{cl1}.
\end{proof}

\noindent (4) Follows from Observation \ref{obsonrk}(1).
\end{proof}

\begin{proposition}
[See {\cite[Conclusion 1.8]{Sh:522}}]
\label{522.1.8}  
\begin{enumerate}
\item ${\rm Pr}^*_{\omega_1}(\beth_{\omega_1})$ holds and hence also ${\rm  
    Pr}_{\omega_1}(\beth_{\omega_1})$. 
\item Assume $\beta<\alpha<\omega_1$, $\bbM$ is a model with a countable
  vocabulary $\tau$ and the universe $\mu$,  $m,n<\omega$, $n>0$, $A\subseteq
  \mu$ and $|A|\geq \beth_{\omega\cdot\alpha}$. Then there is $w\subseteq A$
  with $|w|=n$ and $\rk^*(w,\bbM)\geq \omega\cdot \beta+m$ \footnote{``
    $\cdot$ '' stands for the ordinal multiplication}. 
\end{enumerate}
\end{proposition}

\begin{proof}
(1) Follows from part (2) (and \ref{cl1.7-522}(4)). 
\medskip

\noindent (2) Induction on $\alpha<\omega_1$.
\smallskip

\noindent {\sc Step $\alpha=1$ (and $\beta=0$):}\quad Let
$\bbM,\mu,n,m$ be as in the assumptions, $A\subseteq \mu$ and $|A|\geq
\beth_\omega$. Using the Erd\H{o}s--Rado theorem we may choose a
sequence $\langle a_i:i<\omega_2\rangle$ of distinct elements of $A$ such
that: 
\begin{enumerate}
\item[(a)] the quantifier free type of $\langle a_{i_0},\ldots,a_{i_{m+n}}
  \rangle$ in $\bbM$ is constant for $i_0<\ldots< i_{m+n}<\omega_2$, and  
\item[(b)] for each $k\leq m+n$ the value of $\min\{\omega,
  \rk^*(\{a_{i_0},\ldots, a_{i_{n+m-k}}\}, \bbM)\}$ is constant for
  $i_0<\ldots< i_{m+n-k}<\omega_2$. 
\end{enumerate}
Let $i_\ell=\omega_1\cdot (\ell+1)$ (for $\ell=-1,0,\ldots, n+m$).  Suppose
$\phi(x_0,\ldots,x_{n+m})\in\cL(\tau)$ is a quantifier free formula,
$k\leq n+m$ and $\bbM\models \phi[a_{i_0},\ldots,a_{i_k}, \ldots,
a_{i_{n+m}}]$. It follows from the property stated in (a) above that for
every $i$ in the (uncountable) interval $(i_{k-1},i_k)$ we have $\bbM
\models\varphi[a_{i_0},\ldots,a_{i_{k-1}},a_i, a_{i_{k+1}}, \ldots, a_{i_{m+n}}]$.
Consequently, $\rk^*\big(\{a_{i_0}, \ldots, a_{i_{n+m}}\}, \bbM\big)\geq 0$,
and the homogeneity stated in (b) implies that for every nonempty set
$w\subseteq \omega_2$ with at most $n+m+1$ elements we have 
$\rk^*(\{a_i:i\in w\},\bbM)\geq 0$. Now, by induction on $k\leq m+n$ we will 
argue that
\begin{enumerate}
\item[$(*)_k$] for every nonempty set $w\subseteq\omega_2$ with at most
  $n+m+1-k$ elements we have $\rk^*(\{a_i:i\in w\},\bbM)\geq k$.
\end{enumerate}
We have already justified $(*)_0$. For the inductive step assume $(*)_k$ and
$k< m+n$. Let $i_\ell=\omega_1\cdot (\ell+1)$ and suppose that
$\varphi(x_0,\ldots,x_{m+n-k-1})$ is a quantifier free formula,
$\bbM\models\varphi[a_{i_0},\ldots,a_{i_z}, \ldots, a_{i_{m+n-k-1}}]$ and
$0\leq z\leq n+m-k-1$.  By the homogeneity stated in (a), for every $i$ in
the  uncountable interval $(i_{z-1},i_z)$ we have $\bbM\models
\varphi[a_{i_0},\ldots,a_{i_{z-1}},a_i,a_{i_{z+1}},\ldots, 
a_{i_{m+n-k-1}}]$. The inductive hypothesis $(*)_k$ implies that
$\rk^*(\{a_{i_0},\ldots,a_{i_{z-1}},a_i,a_j,a_{i_{z+1}}, \ldots
a_{i_{m+n-k-1}}\}, \bbM)\geq k$ (for any $i_{z-1}<i<j\leq i_z$). Now we
easily conclude that $k+1\leq \rk^*(\{a_{i_0},\ldots, a_{i_{m+n-k-1}}\},
\bbM)$ and  $(*)_{k+1}$ follows by the homogeneity given by (b).  

Finally note that $(*)_{m+1}$ gives the desired conclusion: taking any
$i_0<\ldots< i_{n-1}<\omega_2$ we will have
$m+1\leq \rk^*\big(\{a_{i_0},\ldots,a_{i_{n-1}}\},\bbM\big)$. 
\medskip

\noindent {\sc Step $\alpha=\gamma+1$:}\quad Let $\bbM,\mu,n,m$ be as in the
assumptions, $A\subseteq \mu$ and $|A|\geq \beth_{\omega\cdot\gamma
  +\omega}$. By the Erd\H{o}s--Rado theorem we may choose a
sequence $\langle a_i:i<\beth_{\omega\cdot\gamma}\rangle$ of distinct
elements of $A$ such that the following two demands are satisfied.  
\begin{enumerate}
\item[(c)] The quantifier free type of $\langle a_{i_0},\ldots,a_{i_{m+n}}
  \rangle$ in $\bbM$ is constant for $i_0<\ldots< i_{m+n}<
  \beth_{\omega\cdot\gamma}$.   
\item[(d)] For each $k\leq m+n$ the value of $\min\{\omega\cdot \alpha,
  \rk^*(\{a_{i_0},\ldots, a_{i_{n+m-k}}\}, \bbM)\}$ is constant for $i_0<\ldots<
  i_{m+n-k}<\beth_{\omega\cdot\gamma}$. 
\end{enumerate}
For any $\ell<\omega$ and $\gamma'<\gamma$, we may apply the inductive
hypothesis to $\{a_i:i<\beth_{\omega\cdot\gamma}\}$, $\ell$, $m+n+1$ and 
$\gamma'$ to find $i_0<\ldots < i_{m+n}<\beth_{\omega\cdot\gamma}$  
such that $\rk^*(\{a_{i_0},\ldots,a_{i_{m+n}}\},\bbM)\geq \omega\cdot
\gamma'+\ell$. By the homogeneity in (d) this implies that 
\begin{enumerate}
\item[$(**)_0$] for all $i_0<\ldots < i_{m+n}<\beth_{\omega\cdot\gamma}$ we
  have  $\rk^*(\{a_{i_0}, \ldots,a_{i_{m+n}}\},\bbM)\geq \omega\cdot
  \gamma$.  
\end{enumerate}
Now, by induction on $k\leq m+n$ we argue that  
\begin{enumerate}
\item[$(**)_k$] for each $i_0<\ldots<i_{m+n-k}<(\beth_{\omega\cdot\gamma})^+$
  we have    
\[\omega\cdot\gamma+k\leq \rk^*(\{a_{i_0},\ldots,a_{i_{m+n-k}}\},\bbM).\]
\end{enumerate}
So assume $(**)_k$, $k<m+n$ and let $i_\ell=\omega_1\cdot (\ell+1)$ (for
$\ell=-1,0, \ldots, m+n$) and $0\leq z\leq n+m-k-1$. Suppose that
$\bbM\models \varphi[a_{i_0},\ldots,a_{i_z}, \ldots, a_{i_{m+n-k-1}}]$.
Then by the homogeneity in (c), for every $i$ in the uncountable interval 
$(i_{z-1},i_z)$ we have $\bbM\models
\varphi[a_{i_0},\ldots,a_{i_{z-1}},a_i,a_{i_{z+1}},\ldots, 
a_{i_{m+n-k-1}}]$. By  the inductive hypothesis $(**)_k$ we know
$\omega\cdot\gamma+k\leq \rk^*(\{a_{i_0},\ldots,a_{i_{z-1}},a_i, a_j,
a_{i_{z+1}},\ldots a_{i_{m+n-k-1}}\},\bbM)$ (for $i_{z-1}<i<j\leq
i_z$). Now we easily conclude that $\omega\cdot\gamma+k+1\leq 
\rk^*(\{a_{i_0},\ldots a_{i_{m+n-k-1}}\},\bbM)$, and $(**)_{k+1}$ follows by
the homogeneity in (d).     

Finally note that $(**)_{m+1}$ gives the desired conclusion: taking any
$i_0<\ldots< i_{n-1}<\beth_{\omega_\gamma}$ we will have
$\rk^*\big(\{a_{i_0},\ldots,a_{i_{n-1}}\},\bbM\big)\geq \omega\cdot
\gamma+m+1$.
\medskip

\noindent {\sc Step $\alpha$ is limit:}\quad should be clear.  
\end{proof}

\begin{definition}
\label{deflam}
Let $\lambda_{\omega_1}$ be the smallest cardinal $\kappa$ such that ${\rm 
  Pr}_{\omega_1}(\kappa)$ and $\lambda^*_{\omega_1}$ be the smallest cardinal 
$\kappa$ such that ${\rm Pr}^*_{\omega_1}(\kappa)$. 
\end{definition}

By Propositions \ref{cl1.7-522}(4) and \ref{522.1.8} we have
$\lambda_{\omega_1}\leq \lambda^*_{\omega_1}\leq \beth_{\omega_1}$.  

\begin{proposition}
[See {\cite[Claim 1.10(1)]{Sh:522}}]
\label{522.1.10}  
If $\bbP$ is a ccc forcing notion and $\lambda$ is a cardinal such that
${\rm Pr}^*_{\omega_1}(\lambda)$ holds, then $\forces_{\bbP}$`` ${\rm
  Pr}^*_{\omega_1}(\lambda)$ and hence also ${\rm Pr}_{\omega_1}(\lambda)$ 
''.  
\end{proposition}

\begin{proof}
Suppose towards contradiction that for some $p\in \bbP$ we have
$p\forces_{\bbP} {\rm NPr}^*_{\omega_1}(\lambda)$. Let
$\tau=\{R_{n,\zeta}:n,\zeta<\omega\}$ where $R_{n,\zeta}$ is an $n$--ary
relation symbol (for $n,\zeta<\omega$).   Then we may pick a name $\name{\bbM}$
for a model on $\lambda$ in vocabulary $\tau$ and an ordinal
$\alpha_0<\omega_1$ such that 
\[\begin{array}{ll}
p\forces&\mbox{``}\name{\bbM}=(\lambda,\{R^{\name{\bbM}}_{n,\zeta}\}_{n,\zeta
          <\omega}) \mbox{ is a model such that}\\
&\mbox{(a)\quad for every $n$ and a quantifier free formula }
  \varphi(x_0,\ldots,x_{n-1})\in \cL(\tau)\\
&\qquad\mbox{ there is $\zeta<\omega$ such that for all }a_0,\ldots, a_{n-1}
    \\ 
&\qquad \ \name{\bbM}\models\varphi[a_0,\ldots,a_{n-1}]\Leftrightarrow
  R_{n,\zeta}[a_0,\ldots, a_{n-1}]\\
&\mbox{(b) }\quad \sup\{\rk(w,\name{\bbM}): \emptyset\neq w\in
  [\lambda]^{<\omega} \}<\alpha_0\mbox{ ''}.
\end{array}\]
Now, let $S_{n,\zeta,\beta,k}$ be an $n$--ary predicate (for
$k<n,\zeta<\omega$ and $-1\leq\beta<\alpha_0$) and let $\tau^*= 
\{S_{n,\zeta,\beta,k}:k<n<\omega,\ \zeta<\omega\mbox{ and
}-1\leq \beta<\alpha_0\}$. (So $\tau^*$ is a countable vocabulary.) We
define a model $\bbM^*$ in the vocabulary $\tau^*$. The universe of $
\bbM^*$ is $\lambda$ and for $k<n$, $\zeta<\omega$ and $-1\leq
\beta<\alpha_0$:  
\[\begin{array}{ll}
S_{n,\zeta,\beta,k}^{\bbM^*}=&\big\{ (a_0,\ldots,a_{n-1})\in {}^n\lambda:
a_0<\ldots<a_{n-1}\mbox{ and}\\
&\qquad\mbox{ some condition $q\geq p$ forces that } \\ 
&\quad \mbox{``} \name{\bbM}\models R_{n,\zeta}[a_0,\ldots,a_{n-1}] \mbox{ 
  and } \rk^*(\{a_0,\ldots, a_{n-1}\}, \name{\bbM})=\beta\mbox{ and }\\
&\quad\ R_{n,\zeta}, k\mbox{ witness that }  \neg \big(\rk^*(\{a_0,\ldots,
  a_{n-1}\},  \name{\bbM})\geq \beta+1 \big)\mbox{ ''} \big\}.
\end{array}\]

\begin{claim}
  \label{cl3}
For every $n$ and every increasing tuple $(a_0,\ldots,a_{n-1})\in
{}^n\lambda$ there are $\zeta<\omega$ and $-1\leq \beta<\alpha_0$ and $k<n$
such that $\bbM^*\models S_{n,\zeta,\beta,k}[a_0,\ldots,a_{n-1}]$.
\end{claim}

\begin{proof}[Proof of the Claim]
  Should be clear.
\end{proof}

\begin{claim}
  \label{cl4}
If $(a_0,\ldots,a_{n-1})\in {}^n\lambda$ and $\bbM^*\models S_{n,
  \zeta,\beta,k}[a_0,\ldots,a_{n-1}]$, then
\[\rk^*\big(\{a_0,\ldots,a_{n-1}\}, \bbM^*\big)\leq \beta.\] 
\end{claim}

\begin{proof}[Proof of the Claim]
First let us deal with the case of $\beta=-1$. Assume towards contradiction
that $\bbM^*\models S_{n, \zeta,-1,k}[a_0, \ldots,a_{n-1}]$,  but
$\rk^*\big(\{a_0,\ldots,a_{n-1}\}, \bbM^*\big)\geq 0$. Then we may find
distinct $\langle b_i:i<\omega_1\rangle \subseteq \lambda\setminus
\{a_0,\ldots,a_{n-1}\}$ such that  
\begin{enumerate}
\item[$(\otimes)_1$] $\bbM^*\models S_{n,\zeta,-1,k} [a_0,\ldots, a_{k-1},
  b_i, a_{k+1}, \ldots, a_{n-1}]$ for all $i<\omega_1$.
\end{enumerate}
For $i<\omega_1$ let $p_i\in\bbP$ be such that $p_i\geq p$ and 
\[\begin{array}{l}
p_i\forces\mbox{`` }\name{\bbM}\models R_{n,\zeta}[a_0,\ldots,b_i,\ldots,
            a_{n-1}] \mbox{ and } \rk^*(\{a_0,\ldots, b_i,\ldots, a_{n-1}\},
            \name{\bbM})=-1\mbox{ and }\\  
\qquad\quad R_{n,\zeta}, k\mbox{ witness that }  \neg
    \big(\rk^*(\{a_0,\ldots, a_{k-1},  b_i, a_{k+1}, \ldots, a_{n-1}\},
    \name{\bbM})\geq 0 \big)\mbox{ ''}     
  \end{array}\]
Let $\name{Y}$ be a name $\bbP$--name such that $p\forces
\name{Y}=\{i<\omega_1: p_i\in\name{G}_\bbP\}$. Since $\bbP$ satisfies ccc,
we may pick $p^*\geq p$ such that  $p^*\forces \mbox{``}\name{Y}$ is
  uncountable ''. Then 
\[p^*\forces\big(\forall i\in\name{Y}\big)\big( \name{\bbM}\models
             R_{n,\zeta}[a_0,\ldots,a_{k-1}, b_i, a_{k+1},\ldots,
             a_{n-1}]\big),\] 
so also 
\[p^*\forces \big\{b<\lambda:\name{\bbM}\models 
R_{n,\zeta}[a_0,\ldots,a_{k-1}, b, a_{k+1},\ldots,
             a_{n-1}]\big\}\mbox{ is uncountable.}\]
But 
\[p^*\forces\big(\forall i\in\name{Y}\big)\big( R_{n,\zeta}, k\mbox{ witness 
}  \neg \big(\rk^*(\{a_0,\ldots,a_{k-1}, b_i, a_{k+1},\ldots, a_{n-1}\},
\name{\bbM})\geq 0 \big)\big),\]
and hence 
\[p^*\forces \big\{b<\lambda:\name{\bbM}\models 
R_{n,\zeta}[a_0,\ldots,a_{k-1}, b, a_{k+1},\ldots,
             a_{n-1}]\big\}\mbox{ is countable },\]
a contradiction. 

Next we continue the proof of the Claim by induction on $\beta<\alpha_0$, so
we assume that $0\leq\beta$ and for $\beta'<\beta$ our claim holds true (for
any $n,\zeta,k$). Assume towards contradiction that $\bbM^*\models S_{n,
  \zeta,\beta,k}[a_0,\ldots,a_{n-1}]$,  but
$\rk^*\big(\{a_0,\ldots,a_{n-1}\}, \bbM^*\big)\geq\beta+1$.  
Then we may find distinct $\langle b_i:i<\omega_1\rangle \subseteq
\lambda\setminus (w\setminus\{a_k\})$ such that 
\begin{enumerate}
\item[$(\oplus)_1$] $\bbM^*\models S_{n,\zeta,\beta,k} [a_0,\ldots, a_{k-1},
  b_i, a_{k+1}, \ldots, a_{n-1}]$ for all $i<\omega_1$, $b_0=a_k$ and 
\item[$(\oplus)_2$] $\rk^*(\{a_0,\ldots, a_{k-1}, b_i,b_j, a_{k+1}, \ldots,
  a_{n-1}\}, \bbM^*)\geq \beta$ for all $i<j<\omega_1$. 
\end{enumerate}
For $i<\omega_1$ let $p_i\in\bbP$ be such that $p_i\geq p$ and 
\[\begin{array}{l}
p_i\forces\mbox{`` }\name{\bbM}\models R_{n,\zeta}[a_0,\ldots,b_i,\ldots,
            a_{n-1}] \mbox{ and } \rk^*(\{a_0,\ldots, b_i,\ldots, a_{n-1}\},
            \name{\bbM})=\beta\mbox{ and }\\  
\ \ R_{n,\zeta}, k\mbox{ witness that }  \neg \big(\rk^*(\{a_0,\ldots,
a_{k-1},  b_i, a_{k+1}, \ldots, a_{n-1}\},  \name{\bbM})\geq \beta+1
  \big)\mbox{ ''}     
  \end{array}\]
Take $p^*\geq p$ such that  
\[p^*\forces \mbox{``}\name{Y}\stackrel{\rm def}{=} \{i<\omega_1:
p_i\in\name{G}_\bbP\}\mbox{ is uncountable ''.}\]   
Since
\[\begin{array}{ll}
p^*\forces &\big(\forall i\in\name{Y}\big)\Big(\ \name{\bbM}\models R_{n,\zeta}[
  a_0,\ldots,a_{k-1}, b_i, a_{k+1},\ldots, a_{n-1}]\ \wedge\ \\
&\qquad\quad R_{n,\zeta}, k\mbox{ witness that }  \neg \big(\rk^*(\{a_0,\ldots,
  b_i,\ldots, a_{n-1}\},  \name{\bbM})\geq \beta+1 \big)\Big),
  \end{array}\]
we see that 
\[p^*\not\forces \big(\forall i,j\in\name{Y}\big)\big( i\neq j \Rightarrow\ 
\rk^*(\{a_0,\ldots,a_{k-1},b_i,b_j,a_{k+1},\ldots, a_{n-1}\},\name{\bbM})
\geq\beta\big).\] 
Consequently we may pick $q\geq p^*$, $i_0,j_0<\omega_1$ and $\gamma<\beta$ 
and $\xi<\omega$ and $\ell\leq n$ such that $b_{i_0}<b_{j_0}$ and  
\[\begin{array}{ll}
q\forces&\mbox{``}  p_{i_0},p_{j_0}\in \name{G}_\bbP\mbox{ and }
  \rk^*(\{a_0,\ldots, a_{k-1}, b_{i_0},b_{j_0}, a_{k+1},\ldots, a_{n-1}\},
          \name{\bbM})=\gamma\mbox{ and }\\ 
&\quad\ R_{n+1,\xi}\mbox{ and }\ell \mbox{ witness that }\\
&\quad \neg \big(\rk^*(\{a_0,\ldots, a_{k-1}, b_{i_0},b_{j_0},
  a_{k+1},\ldots, a_{n-1}\},  \name{\bbM})\geq \gamma+1 \big)\mbox{ ''.}       
  \end{array}\]
Then $\bbM^*\models S_{n+1,\xi,\ell,\gamma}[a_0,\ldots, a_{k-1}, b_{i_0},
b_{j_0}, a_{k+1},\ldots, a_{n-1}]$ and by the inductive hypothesis 
$\rk^*(\{a_0,\ldots, a_{k-1}, b_{i_0},b_{j_0}, a_{k+1},\ldots, a_{n-1}\},
\name{\bbM})\leq \gamma$, contradicting clause $(\oplus)_2$ above.     
\end{proof}
\end{proof}

\begin{corollary}
 \label{lamCoh}
Let $\mu=\beth_{\omega_1}\leq\kappa$ and $\bbC_\kappa$ be the forcing notion
adding $\kappa$ Cohen reals. Then
$\forces_{\bbC_\kappa}\lambda_{\omega_1}\leq\mu\leq\con$.  
\end{corollary}

\section{Spectrum of translation non-disjointness}

\begin{definition}
\label{otps}
Let $B\subseteq \can$ and $1\leq k\leq \con$. 
\begin{enumerate}
\item We say that $B$ is {\em perfectly orthogonal to $k$--small\/} (or a
$k$--{\bf pots}--set) if there is a perfect set $P\subseteq \can$ such that 
$|(B+x)\cap (B+y)|\geq k$ for all $x,y\in P$. \\
The set $B$ is {\em a $k$--{\bf npots}--set\/} if it is not $k$--{\bf pots}.  
\item We say that $B$ has {\em $\lambda$ many pairwise $k$--nondisjoint
    translations\/} if for some set $X\subseteq\can$ of cardinality
  $\lambda$, for all $x,y\in X$ we have $\big|(B+x)\cap (B+y)\big|\geq k$. 
\item We define the {\em spectrum of translation $k$--non-disjointness of
    $B$} as  
\[\stnd_k(B)=\{(x,y)\in\can\times\can: |(B+x)\cap (B+y)|\geq k\}.\] 
\end{enumerate} 
\end{definition}

\begin{remark}
  \begin{enumerate}
\item Note that if $B\subseteq \can$ is an uncountable Borel set, then there is a 
perfect set $P\subseteq B$. For $B,P$ as abovefor every $x,y\in P$ we have
$0=x+x=y+y \in  (B+x) \cap (B+y)$ and $x+y\in (B+x) \cap
(B+y)$. Consequently every  uncountable Borel subset of $\can$ is a
$2$--{\bf pots}--set.   
\item Assume $B\subseteq \can$ and $x,y\in\can$. If $b_x,b_y\in B$ and
  $b_x+x=b_y+y\in (B+x)\cap (B+y)$, then also $b_x+y=b_y+x\in (B+x)\cap
  (B+y)$. Consequently, if $(B+x)\cap (B+y)\neq\emptyset$ is finite, then it
  has an even number of elements.   
  \end{enumerate}
\end{remark}

\begin{proposition}
\label{proptostart}
\begin{enumerate}
\item Let $1\leq k\leq\con$. A set $B\subseteq \can$ is a $k$--{\bf
    pots}--set if and only if there  is a perfect set $P\subseteq\can$ such
  that $P\times P\subseteq\stnd_k(B)$.
\item Assume $k<\omega$. If $B$ is $\Sigma^0_2$, then $\stnd_k(B)$ is
  $\Sigma^0_2$ as well. If $B$ is Borel, then $\stnd_k(B)$ and
  $\stnd_\omega(B)$ are $\Sigma^1_1$ and $\stnd_\con(B)$ is $\Delta^1_2$.  
\item Let $\con<\kappa\leq\mu$ and let $\bbC_\mu$ be the forcing notion
  adding $\mu$  Cohen reals. Then, remembering Definition \ref{otps}(2),  
\[\begin{array}{l}
\forces_{\bbC_\mu}\mbox{`` if a Borel set $B\subseteq\can$ has $\kappa$ many
  pairwise $k$--non-disjoint translates,}\\
\qquad\quad\mbox{then $B$ is an $k$--{\bf pots}--set ''}.
\end{array}\]
\item If $k<\omega$, $B$ is a (code for) $\Sigma^0_2$ $k$--{\bf npots}--set
  and $\bbP$ is a forcing notion, then $\forces_\bbP$`` $B$ is a (code for)
  $k$--{\bf npots}--set ''. 
\item Assume ${\rm Pr}_{\omega_1}(\lambda)$. If $k\leq\omega$ and a Borel
  set $B\subseteq\can$ has $\lambda$ many pairwise $k$--nondisjoint
  translates, then it is a $k$--{\bf pots}--set. 
\end{enumerate}
\end{proposition}

\begin{proof}
(2)\quad Let $B=\bigcup\limits_{n<\omega} F_n$, where each $F_n$ is a closed
subset of $\can$. Then 
\[\begin{array}{l}
(x,y)\in \stnd_k(B)\Leftrightarrow\\
\qquad \big(\exists n_0,\ldots, n_{k-1},m_0,\ldots,m_{k-1},N<\omega
    \big)\big(\exists z_0,\ldots,z_{k-1}\in\can\big) \big(\forall i,j<k\big) \Big(\\
\qquad\qquad\qquad\qquad 
(i\neq j \Rightarrow z_i\rest N\neq z_j\rest N)\ \wedge \
    z_i+x\in F_{n_i}\ \wedge \ z_i+y\in F_{m_i}\Big)  
\end{array}\]
The formula $\big(\forall i,j<k\big) \big((i\neq j \Rightarrow
z_i\rest N\neq z_j\rest N)\ \wedge \ z_i+x\in F_{n_i}\ \wedge \ z_i+y\in
F_{m_i} \big)$ represents a compact subset of $\big(\can\big)^{k+2}$
and hence easily the assertion follows.  
\medskip

\noindent (3)\quad This is a consequence of (1,2) above and Shelah
\cite[Fact 1.16]{Sh:522}. 
\medskip

\noindent (4)\quad If $B$ is a $\Sigma^0_2$ set then the formula ``there is a
perfect set $P\subseteq \can$ such that for all $x,y\in P$ we have $(x,y)\in 
\stnd_k(B)$ '' is $\Sigma^1_2$ (remember (2) above). 
\medskip

\noindent (5)\quad By \cite[Claim 1.12(1)]{Sh:522}. 
\end{proof}
\bigskip

We want to analyze $k$--{\bf pots}--sets in more detail, restricting
ourselves to $\Sigma^0_2$ subsets of $\can$. For the rest of this
section we assume the following Hypothesis.

\begin{hypothesis}
\label{hyp}
  \begin{enumerate}
\item $T_n\subseteq {}^{\omega>} 2$ is a tree with no maximal nodes
    (for $n<\omega$);
\item $B=\bigcup\limits_{n<\omega} \lim(T_n)$, $\bar{T}=\langle T_n:
  n<\omega \rangle$;
\item $2\leq\iota<\omega$, $k=2\iota$. 
  \end{enumerate}
\end{hypothesis}

\begin{definition}
  \label{mtkDef}
Let $\Mtk$ consist of all tuples 
\[\bm=(\ell_\bm,u_\bm,\bar{h}_\bm,\bar{g}_\bm)=(\ell,u,\bar{h},\bar{g})\]
such that:
\begin{enumerate}
\item[(a)] $0<\ell<\omega$, $u\subseteq {}^\ell 2$ and $2\leq |u|$;
\item[(b)] $\bar{h}=\langle h_i:i<\iota\rangle$, $\bar{g}=\langle g_i:i<
  \iota\rangle$ and for each $i<\iota$ we have 
\[h_i:u^{\langle 2\rangle}\longrightarrow \omega\quad\mbox{ and }\quad 
    g_i:u^{\langle 2\rangle} \longrightarrow \bigcup_{n<\omega}( T_n\cap
    {}^\ell 2);\]
\item[(c)] $g_i(\eta,\nu)\in T_{h_i(\eta,\nu)}\cap {}^\ell 2$ for all
  $(\eta,\nu)\in u^{\langle 2\rangle}$, $i<\iota$;
\item[(d)] if $(\eta,\nu)\in u^{\langle 2\rangle}$ and $i<\iota$, then
  $\eta+ g_i(\eta,\nu) =\nu+ g_i(\nu,\eta)$;  
\item[(e)] for any $(\eta,\nu)\in u^{\langle 2\rangle}$, there are no
  repetitions in the sequence $\langle g_i(\eta,\nu),g_i(\nu,\eta):
  i<\iota\rangle$. 
\end{enumerate}
\end{definition}

\begin{definition}
\label{traDef}
Assume $\bm=(\ell,u,\bar{h},\bar{g})\in \Mtk$ and $\rho\in {}^\ell 2$. We
define $\bm+\rho=(\ell',u',\bar{h}',\bar{g}')$ by
\begin{itemize}
\item $\ell'=\ell$, $u'=\{\eta+\rho:\eta\in u\}$,
\item $\bar{h}'=\langle h'_i:i<\iota\rangle$ where $h'_i:(u')^{\langle
    2\rangle} \longrightarrow \omega$ are such that
  $h'_i(\eta+\rho,\nu+\rho)=h_i(\eta,\nu)$ for $(\eta,\nu)\in u^{\langle
    2\rangle}$, 
\item $\bar{g}'=\langle g'_i:i<\iota\rangle$ where $g'_i:(u')^{\langle 2\rangle}
  \longrightarrow \bigcup\limits_{n<\omega} (T_n\cap {}^\ell 2)$ are such
  that $g'_i (\eta+\rho, \nu+\rho)=g_i(\eta,\nu)$ for $(\eta,\nu)\in
  u^{\langle 2\rangle}$.
\end{itemize}
Also if $\rho\in\can$, then we set $\bm+\rho=\bm+(\rho\rest\ell)$. 
\end{definition}

\begin{observation}
  \begin{enumerate}
\item If $\bm\in \Mtk$ and $\rho\in {}^{\ell_\bm}2$, then $\bm+\rho
  \in\Mtk$. 
\item For each $\rho\in\can$ the mapping
\[\Mtk\longrightarrow\Mtk:\bm\mapsto\bm+\rho\]
is a bijection.
  \end{enumerate}
\end{observation}

\begin{definition}
\label{extDef}
Assume $\bm,\bn\in\Mtk$. We say that {\em $\bn$ extends $\bm$\/}
($\bm\sqsubseteq \bn$ in short) if and only if:
\begin{itemize}
\item $\ell_\bm\leq \ell_\bn$, $u_\bm=\{\eta\rest\ell_\bm:\eta\in u_\bn\}$,
  and 
\item for every $(\eta,\nu)\in (u_\bn)^{\langle 2\rangle }$ such that
  $\eta\rest \ell_\bm \neq \nu\rest\ell_\bm$ and each $i<\iota$ we have  
\[h^\bm_i(\eta\rest \ell_\bm,\nu\rest \ell_\bm)= h^\bn_i(\eta,\nu)\quad
  \mbox{ and }\quad g^\bm_i(\eta\rest \ell_\bm,\nu\rest \ell_\bm)=
  g^\bn_i(\eta,\nu)\rest \ell_\bm.\] 
\end{itemize}
\end{definition}

\begin{definition}
\label{ndrkdef}
We define a function $\ndrk:\Mtk\longrightarrow {\rm ON}\cup\{\infty\}$
declaring inductively when $\ndrk(\bm)\geq\alpha$ (for an ordinal
$\alpha$). 
\begin{itemize}
\item $\ndrk(\bm)\geq 0$ always;
\item if $\alpha$ is a limit ordinal, then 
\[\ndrk(\bm)\geq \alpha\Leftrightarrow (\forall\beta<\alpha)(\ndrk(\bm)\geq
  \beta);\]
\item if $\alpha=\beta+1$, then $\ndrk(\bm)\geq \alpha$ if and only if for
  every $\nu\in u_\bm$ there is $\bn\in \Mtk$ such that
  $\ell_\bn>\ell_\bm$, $\bm\sqsubseteq\bn$ and $\ndrk(\bn)\geq \beta$ and 
\[|\{\eta\in u_\bn:\nu\vtl\eta\}|\geq 2;\] 
\item $\ndrk(\bm)=\infty$ if and only if $\ndrk(\bm)\geq \alpha$ for all
  ordinals $\alpha$. 
\end{itemize}
We also define 
\[\NDRK(\bar{T})=\sup\{\ndrk(\bm)+1:\bm\in\Mtk\}.\]
\end{definition}

\begin{lemma}
\label{lemonrk}
  \begin{enumerate}
\item The relation $\sqsubseteq$ is a partial order on $\Mtk$.
\item If $\bm,\bn\in\Mtk$ and $\bm\sqsubseteq\bn$ and $\alpha\leq
  \ndrk(\bn)$, then $\alpha\leq\ndrk(\bm)$. 
\item The function $\ndrk$ is well defined.
\item If $\bm\in\Mtk$ and $\rho\in\can$ then $\ndrk(\bm)=\ndrk(\bm+\rho)$. 
\item If $\bm\in\Mtk$, $\nu\in u_\bm$ and $\ndrk(\bm)\geq \omega_1$, then
  there is an $\bn\in\Mtk$ such that $\bm\sqsubseteq\bn$,
  $\ndrk(\bn)\geq\omega_1$, and  
\[|\{\eta\in u_\bn:\nu\vtl\eta\}|\geq 2.\]
\item If $\bm\in\Mtk$ and $\infty>\ndrk(\bm)=\beta>\alpha$, then there is
  $\bn\in \Mtk$ such that $\bm\sqsubseteq \bn$ and $\ndrk(\bn)=\alpha$. 
\item If $\NDRK(\bar{T})\geq \omega_1$, then $\NDRK(\bar{T})=\infty$. 
\item Assume $\bm\in\Mtk$ and $u'\subseteq u_\bm$, $|u'|\geq 2$. Put
  $\ell'=\ell_m$, $h_i'=h_i^\bm\rest u^{\langle 2\rangle}$ and $g_i'=g_i^\bm
  \rest u^{\langle 2\rangle}$ (for $i<\iota$), and let  $\bm\rest
  u'=(\ell',u', \bar{h}', \bar{g}')$. Then $\bm\rest u'\in \Mtk$
  and $\ndrk(\bm)\leq \ndrk(\bm\rest u')$. 
  \end{enumerate}
\end{lemma}

\begin{proof}
(1)\quad Should be clear.
\medskip

\noindent (2)\quad Induction on $\alpha$. If $\alpha=\alpha_0+1$ and
$\bn'\sqsupseteq\bn$ is one of the witnesses used to claim that
$\ndrk(\bn)\geq \alpha_0+1$, then this $\bn'$ can also be used for
$\bm$. Hence we can argue the successor step of the induction. The limit
steps are even easier. 
\medskip

\noindent  (3)\quad One has to show that if $\beta<\alpha$ and
  $\ndrk(\bm)\geq \alpha$, then $\ndrk(\bm)\geq \beta$.  This can be shown
  by induction on $\alpha$: at the successor stage if $\bn$ is one of the
  witnesses used to claim that $\ndrk(\bm)\geq \alpha+1$, then
  $\ndrk(\bn)\geq\alpha$. By (2) we get $\ndrk(\bm)\geq\alpha$ and by the
  inductive hypothesis  $\ndrk(\bm)\geq \gamma$ for $\gamma\leq \alpha$.  
Limit stages should be  clear too.  \medskip

\noindent (4)\quad Should be clear. 
\medskip

\noindent (5)\quad Let $\cN$ be the collection of all $\bn\in\Mtk$ such that
$\bm\sqsubseteq\bn$ and $|\{\eta\in u_\bn:\nu\vtl\eta\}|\geq 2$. If
$\ndrk(\bn_0)\geq \omega_1$ for some $\bn_0\in\cN$, then we are done. So
suppose towards contradiction that there is no such $\bn_0$. Then, as $\cN$
is countable, 
\[\alpha_0\stackrel{\rm def}{=} \sup\{\ndrk(\bn)+1:\bn\in\cN\}<\omega_1.\]
But $\ndrk(\bm)\geq\alpha_0+1$ implies that $\ndrk(\bn_1)\geq\alpha_0$ for
some $\bn_1\in\cN$, a contradiction. 
\medskip

\noindent (6)\quad Induction on ordinals $\beta$ (for all $\alpha<
\beta$). The main point is that if $\ndrk(\bm)=\beta$, then for some
$\nu\in u_\bm$ we cannot find $\bn$ as needed for witnessing $\ndrk(\bm)
\geq\beta+1$, but for each $\gamma<\beta$ we can find $\bn$  needed for
$\ndrk(\bm)\geq\gamma+1$. Therefore for each $\gamma<\beta$ we may find
$\bn\sqsupseteq \bm$ such that $\gamma\leq\ndrk(\bn)<\beta$. 
\medskip

\noindent (7)\quad Follows from (6) above.
\medskip

\noindent (8)\quad It should clear that $(\ell',u',\bar{h}', \bar{g}')\in
\Mtk$. Also, by a straightforward induction on $\alpha$ for all $\bm$ and 
restrictions $\bm\rest u'$, one shows that  
\[\alpha\leq \ndrk(\bm)\ \Rightarrow\ \alpha\leq \ndrk(\bm\rest u').\]   
\end{proof}

\begin{proposition}
\label{eqnd}
The following conditions are equivalent.
\begin{enumerate}
\item[(a)] $\NDRK(\bar{T})\geq \omega_1$.
\item[(b)] $\NDRK(\bar{T})=\infty$.
\item[(c)] There is a perfect set $P\subseteq\can$ such that 
\[\big(\forall \eta,\nu\in P \big) \big (|(B+\eta)\cap (B+\nu)|\geq k 
\big).\]
\item[(d)] In some ccc forcing extension, there is $A\subseteq\can$ of 
  cardinality $\lambda_{\omega_1}$ such that 
\[\big(\forall \eta,\nu\in A\big) \big (|(B+\eta)\cap (B+\nu)|\geq k 
\big).\]
\end{enumerate}
\end{proposition}

\begin{proof}
${\rm (a)} \Rightarrow {\rm (b)}$\quad This is Lemma \ref{lemonrk}(7).
\medskip

\noindent ${\rm (b)} \Rightarrow {\rm (c)}$\quad If $\NDRK(\bar{T})=\infty$
then there is $\bm_0\in \Mtk$ with $\ndrk(\bm_0)\geq \omega_1$. Using Lemma
\ref{lemonrk}(5) we may now choose a sequence $\langle \bm_j:j<\omega\rangle
\subseteq \Mtk$ such that for each $j<\omega$:
\begin{enumerate}
\item[(i)] $\bm_j\sqsubseteq \bm_{j+1}$,
\item[(ii)] $\ndrk(\bm_j)\geq \omega_1$,
\item[(iii)] $|\{\eta\in u_{\bm_{j+1}}:\nu\vtl \eta|\geq 2$ for each $\nu\in
  u_{\bm_j}$. 
\end{enumerate}
Let $P=\{\rho\in\can:(\forall j<\omega)(\rho\rest \ell_{\bm_j}\in
u_{\bm_j})\}$. Clearly, $P$ is a perfect set. For $\eta,\nu\in P$,
$\eta\neq\nu$, let $j_0$ be the smallest such that $\eta\rest
\ell_{\bm_{j_0}} \neq \nu\rest \ell_{\bm_{j_0}}$ and let 
\[G_i(\eta,\nu)=\bigcup\big\{g_i^{\bm_j}(\eta\rest\ell_{\bm_j},
\nu\rest \ell_{\bm_j}): j\geq j_0\big\}\in\lim \Big(T_{h^{\bm_{j_0}}_i(
  \eta\rest\ell_{\bm_{j_0}}, \nu\rest\ell_{\bm_{j_0}})}\Big) \quad \mbox{
  for }i<\iota.\]  
Then $G_i:P^{\langle 2\rangle}\longrightarrow B$ and for $(\eta,\nu)\in
  P^{\langle 2\rangle}$ and $i<\iota$:
\[\eta+G_i(\eta,\nu)=\nu+G_i(\nu,\eta)\quad \mbox{ and }\quad 
\eta+G_i(\nu,\eta)=\nu+G_i(\eta,\nu).\]
Moreover, there are no repetitions in the sequence $\langle G_i(\eta,\nu),
G_i(\nu,\eta): i<\iota\rangle$. Hence, for distinct $\eta,\nu\in P$ we have
$|(B+\eta)\cap (B+\nu)|\geq 2\iota=k$. 
\medskip

\noindent ${\rm (c)} \Rightarrow {\rm (d)}$\quad Assume (c). Let
$\kappa=\beth_{\omega_1}$. By Corollary \ref{lamCoh} we know that
$\forces_{\bbC_\kappa} \lambda_{\omega_1}\leq \con$. Remembering Proposition
\ref{proptostart}(1,2), we note that the formula ``$P\times P\subseteq
\stnd_k(B)$'' is $\Pi^1_1$, so it holds in the forcing extension by
$\bbC_\kappa$. Now we easily conclude (d). 
\medskip

\noindent ${\rm (d)} \Rightarrow {\rm (a)}$\quad Assume (d) and let $\bbP$
be the ccc forcing notion witnessing this assumption, $G\subseteq\bbP$ be
generic over $\bV$. Let us work in $\bV[G]$. 

Let $\langle \eta_\alpha:\alpha<\lambda_{\omega_1}\rangle$ be
a sequence of distinct elements of $\can$ such that
\[\big(\forall \alpha<\beta<\lambda_{\omega_1}\big)\big(|(B+\eta_\alpha)\cap 
(B+\eta_\beta)|\geq k\big).\]
Let $\tau=\{R_\bm:\bm\in\Mtk\}$ be a (countable) vocabulary where each
$R_\bm$ is a $|u_\bm|$--ary relational symbol. Let
$\bbM=\big(\lambda_{\omega_1}, \big\{R^\bbM_\bm\big\}_{\bm\in \Mtk}\big)$ be
the model in the vocabulary $\tau$, where for $\bm=(\ell,u,\bar{h},
\bar{g})\in\Mtk$  the relation $R_\bm^\bbM$ is defined by    
\[\begin{array}{ll}
R^\bbM_\bm=&\Big\{(\alpha_0,\ldots,\alpha_{|u|-1})\in
(\lambda_{\omega_1})^{|u|}:\{\eta_{\alpha_0}\rest \ell,\ldots,
        \eta_{|u|-1}\rest \ell\} =u \mbox{ and}\\
&\qquad \mbox{for distinct }j_1,j_2<|u|\mbox{ there are }
  G_i(\alpha_{j_1},\alpha_{j_2})\mbox{ (for $i<\iota$) such that}\\
&\qquad  g_i(\eta_{\alpha_{j_1}}\rest \ell,\eta_{\alpha_{j_2}}\rest \ell)
  \vtl G_i(\alpha_{j_1},\alpha_{j_2})\in \lim\big(
  T_{h_i(\eta_{\alpha_{j_1}}\rest \ell,\eta_{\alpha_{j_2}}\rest \ell)} \big)
  \mbox{ and}\\
&\qquad\eta_{\alpha_{j_1}}+G_i(\alpha_{j_1},\alpha_{j_2}) =
  \eta_{\alpha_{j_2}}+ G_i(\alpha_{j_2}, \alpha_{j_1})\ \Big\}.
\end{array}\]

\begin{claim}
\label{cl5}
\begin{enumerate}
\item If $\alpha_0,\alpha_1,\ldots,\alpha_{j-1}<\lambda_{\omega_1}$ are
  distinct, $j\geq 2$, then for sufficiently large $\ell<\omega$ there is
  $\bm\in \Mtk$ such that  
\[\ell_\bm=\ell,\quad u_\bm=\{\eta_{\alpha_0}\rest \ell, \ldots,
\eta_{\alpha_{j-1}}\rest \ell\}\quad \mbox{ and }\quad \bbM\models
R_\bm[\alpha_0,\ldots,\alpha_{j-1}].\]
\item Assume that  $\bm\in\Mtk$, $j<|u_{\bm_0}|$, 
  $\alpha_0,\alpha_1,\ldots,\alpha_{|u_{\bm}|-1} < \lambda_{\omega_1}$ and
  $\alpha^*< \lambda_{\omega_1}$ are all pairwise distinct and such that
  $\bbM\models R_{\bm}[\alpha_0,\ldots,\alpha_j, \ldots, \alpha_{|u_{\bm}|-1}]$
  and  $\bbM\models R_{\bm}[\alpha_0,\ldots,\alpha_{j-1},\alpha^*,\alpha_{j+1},
\ldots \alpha_{|u_{\bm}|-1}]$. Then for every sufficiently large
$\ell>\ell_{\bm}$ there is $\bn\in \Mtk$ such that $\bm\sqsubseteq \bn$ 
and  
\[\ell_\bn=\ell,\quad u_\bn=\{\eta_{\alpha_0}\rest \ell, \ldots,
\eta_{\alpha_{|u_\bm|-1}}\rest \ell,\eta_{\alpha^*}\rest \ell\}\quad \mbox{
  and 
}\quad \bbM\models R_\bn[\alpha_0,\ldots,\alpha_{|u_\bm|-1},\alpha^*].\]
\item If $\bm\in\Mtk$ and $\bbM\models R_\bm[\alpha_0,\ldots,
  \alpha_{|u_\bm|-1}]$, then  
\[\rk(\{\alpha_0,\ldots,\alpha_{|u_\bm|-1}\},\bbM)\leq \ndrk(\bm).\]  
\end{enumerate}
\end{claim}

\begin{proof}[Proof of the Claim]
  (1)\quad For  distinct $j_1,j_2<j$ let $G_i(\alpha_{j_1},\alpha_{j_2})\in
  B$ (for $i<\iota$) be such that 
\[\eta_{\alpha_{j_1}}+G_i(\alpha_{j_1},\alpha_{j_2}) =
  \eta_{\alpha_{j_2}}+ G_i(\alpha_{j_2}, \alpha_{j_1})\]
and there are no repetitions in the sequence $\langle
G_i(\alpha_{j_1},\alpha_{j_2}), G_i(\alpha_{j_2},\alpha_{j_1}):
i<\iota\rangle$. Suppose that $\ell<\omega$ is such that for any distinct
$j_1,j_2<j$ we have $\eta_{\alpha_{j_1}}\rest \ell \neq
\eta_{\alpha_{j_2}}\rest \ell$ and there are no repetitions in the sequence
$\langle G_i(\alpha_{j_1},\alpha_{j_2})\rest \ell,
G_i(\alpha_{j_2},\alpha_{j_1})\rest \ell:i<\iota\rangle$. Now let 
$u=\{\eta_{\alpha_{j'}}\rest \ell: j'<j\}$, and for $i<\iota$ let
$g_i(\eta_{\alpha_{j_1}}\rest \ell,\eta_{\alpha_{j_2}}\rest
\ell)=G_i(\alpha_{j_1},\alpha_{j_2})\rest \ell$, and let
$h_i(\eta_{\alpha_{j_1}}\rest \ell,\eta_{\alpha_{j_2}}\rest \ell) <\omega$
be such that $G_i(\alpha_{j_1},\alpha_{j_2})\in
\lim\big(T_{h_i(\eta_{\alpha_{j_1}}\rest \ell,\eta_{\alpha_{j_2}}\rest\ell)}
\big)$.  It should be clear that this way we defined $\bm=(\ell,u,\bar{h},
\bar{g})\in \Mtk$ and $\bbM\models R_\bm[\alpha_0,\ldots, \alpha_{j-1}]$. 
\medskip

\noindent (2)\quad An obvious modification of the argument above. 
\medskip

\noindent (3)\quad By induction on $\beta$ we show that {\em for every\/}
$\bm\in\Mtk$ and {\em all\/} $\alpha_0,\ldots,\alpha_{|u_\bm|-1}<
\lambda_{\omega_1}$ such that $\bbM\models R_\bm[\alpha_0,\ldots,
\alpha_{|u_\bm|-1}]$: 
\begin{quotation}
$\beta\leq \rk(\{\alpha_0,\ldots,\alpha_{|u_\bm|-1}\}, \bbM)$ implies $\beta\leq  
\ndrk(\bm)$. 
\end{quotation}
\smallskip

\noindent {\sc Steps $\beta=0$ and $\beta$ is limit:}\quad should be clear. 
\smallskip

\noindent {\sc Step $\beta=\gamma+1$:}\quad Suppose $\bm\in\Mtk$ and 
$\alpha_0,\ldots,\alpha_{|u_\bm|-1}<\lambda_{\omega_1}$ are such that
$\bbM\models R_\bm[\alpha_0,\ldots, \alpha_{|u_\bm|-1}]$ and
$\gamma+1\leq \rk(\{\alpha_0,\ldots,\alpha_{|u_\bm|-1}\}, \bbM)$. Let
$\nu\in u_\bm$, so $\nu=\eta_{\alpha_j}\rest \ell_\bm$ for some $j<|u_\bm|$.
Since $\gamma+1\leq \rk(\{\alpha_0,\ldots,\alpha_{|u_\bm|-1}\}, \bbM)$ we may
find $\alpha^*\in \lambda_{\omega_1}\setminus\{\alpha_0,\ldots,
\alpha_{|u_\bm|-1}\}$ such that $\bbM\models R_\bm[\alpha_0,\ldots,
\alpha_{j-1},\alpha^*, \alpha_{j+1}, \ldots,\alpha_{|u|-1}]$ and
$\rk(\{\alpha_0,\ldots,\alpha_{|u|-1}, \alpha^*\}, \bbM)\geq \gamma$. Taking
sufficiently large $\ell$ we may use clause (2) to find $\bn\in\Mtk$ such
that $\bm\sqsubseteq \bn$, $\ell_\bn=\ell$ and $\bbM\models
R_\bn[\alpha_0,\ldots,\alpha_{|u_\bm|-1},\alpha^*]$ and $|\{\eta\in
u_\bn:\nu\vtl\eta\}|\geq 2$. By the inductive hypothesis we have also
$\gamma\leq \ndrk(\bn)$.  Now we may easily conclude that $\gamma+1\leq
\ndrk(\bm)$. 
\end{proof}

By the definition of $\lambda_{\omega_1}$, 
\begin{enumerate}
\item[$(\odot)$] $\sup\{\rk(w,\bbM):\emptyset\neq w\in
  [\lambda_{\omega_1}]^{<\omega} \}\geq\omega_1$ 
\end{enumerate}
Now, suppose that $\beta<\omega_1$. By $(\odot)$, there are distinct 
$\alpha_0,\ldots,\alpha_{j-1}<\lambda_{\omega_1}$, $j\geq 2$, such that
$\rk(\{\alpha_0,\ldots, \alpha_{j-1}\},\bbM)\geq \beta$. By Claim \ref{cl5}(1)
we may find $\bm\in\Mtk$ such that $\bbM\models
R_\bm[\alpha_0,\ldots,\alpha_{j-1}]$. Then by Claim \ref{cl5}(3) we also
have $\ndrk(\bm)\geq \beta$. Consequently, $\NDRK(\bar{T})\geq \omega_1$.  

All the considerations above where carried out in $\bV[G]$. However, the
rank function $\ndrk$ is absolute, so we may also claim that in $\bV$ we have
$\NDRK(\bar{T})\geq \omega_1$.   
\end{proof}

\begin{corollary}
\label{corlar}
Assume that $\vare\leq\omega_1$ and ${\rm Pr}_\vare(\lambda)$.  If there is
$A\subseteq\can$ of  cardinality $\lambda$ such that 
\[\big(\forall \eta,\nu\in A\big) \big (|(B+\eta)\cap (B+\nu)|\geq k 
\big),\]
then $\NDRK(\bar{T})\geq\vare$.
\end{corollary}

\begin{proof}
This is essentialy shown by the proof of the implication ${\rm (d)}
\Rightarrow {\rm (a)}$ of Proposition \ref{eqnd}. 
\end{proof}

\section{The forcing}

In this  section we construct a forcing notion adding a sequence
$\bar{T}$ of subtrees of ${}^{\omega>} 2$ such that
$\NDRK(\bar{T})<\omega_1$. The sequence $\bar{T}$ will be added by 
finite approximations, so it will be convenient to have finite version of
Definition \ref{mtkDef}.

\begin{definition}
  \label{fmtkDef}
Assume that 
\begin{itemize}
\item $2\leq\iota<\omega$, $k=2\iota$, and $0<n,M<\omega$, 
\item $\bar{t}=\langle t_m:m<M\rangle$, and  each $t_m$ is a subtree
  of ${}^{n\geq} 2$ in which all terminal  branches are of length $n$,
\item $T_j\subseteq {}^{\omega>} 2$ (for $j<\omega$) are trees with no
  maximal nodes, $\bar{T}=\langle T_j:j<\omega\rangle$ and $t_m=T_m\cap
  {}^{n\geq} 2$ for $m<M$,
\item $\Mtk$ is defined as in Definition \ref{mtkDef}.   
\end{itemize}
\begin{enumerate}
\item Let $\fMtk$ consist of all tuples
  $\bm=(\ell_\bm,u_\bm,\bar{h}_\bm,\bar{g}_\bm)\in \Mtk$ such that
  $\ell_\bm\leq n$ and $\rng(h^\bm_i)\subseteq M$ for each $i<\iota$. 
\item Assume $\bm,\bn\in\fMtk$. We say that {\em $\bm$, $\bn$ are
    essentially the same\/} ($\bm\doteqdot \bn$ in short) if and only if:
  \begin{itemize}
\item $\ell_\bm=\ell_\bn$, $u_\bm=u_\bn$ and 
\item for each $(\eta,\nu)\in (u_\bm)^{\langle 2\rangle}$ we have 
\[\big\{\{g_i^\bm(\eta,\nu), g_i^\bm(\nu,\eta)\}:i<\iota\big\} 
=\big\{\{g^\bn_i(\eta,\nu), g^\bn_i(\nu,\eta)\}:i<\iota\big\},\]
and for $i,j<\iota$:\\
if $g_i^\bm(\eta,\nu)= g^\bn_j(\eta,\nu)$, then $h_i^\bm(\eta,\nu)  
=h^\bn_j(\eta,\nu)$,\\
if $g_i^\bm(\eta,\nu)= g^\bn_j(\nu,\eta)$, then
$h_i^\bm(\eta,\nu)=h^\bn_j(\nu,\eta)$. 
\end{itemize}
\item Assume $\bm,\bn\in\fMtk$. We say that {\em $\bn$ essentially extends 
    $\bm$\/} ($\bm\sqsubseteq^* \bn$ in short) if and only if: 
  \begin{itemize}
\item $\ell_\bm\leq \ell_\bn$, $u_\bm=\{\eta\rest\ell_\bm:\eta\in u_\bn\}$,
  and 
\item for every $(\eta,\nu)\in (u_\bn)^{\langle 2\rangle }$ such that
  $\eta\rest \ell_\bm \neq \nu\rest\ell_\bm$ we have  
\[\big\{\{g_i^\bm(\eta\rest\ell_\bm,\nu\rest \ell_\bm), g_i^\bm(\nu\rest
  \ell_\bm, \eta\rest\ell_\bm)\}:i<\iota\big\}
  =\big\{\{g^\bn_i(\eta,\nu)\rest \ell_\bm, g^\bn_i(\nu,\eta)\rest
  \ell_\bm\}:i<\iota\big\},\]
and for $i,j<\iota$:\\
if $g_i^\bm(\eta\rest\ell_\bm,\nu\rest \ell_\bm)= g^\bn_j(\eta,\nu)\rest
\ell_\bm$, then $h_i^\bm(\eta\rest\ell_\bm,\nu\rest \ell_\bm)
=h^\bn_j(\eta,\nu)$,\\
if $g_i^\bm(\eta\rest\ell_\bm,\nu\rest \ell_\bm)= g^\bn_j(\nu,\eta)\rest
\ell_\bm$, then $h_i^\bm(\eta\rest\ell_\bm,\nu\rest \ell_\bm)=h^\bn_j(\nu,
\eta)$. 
\end{itemize}
\end{enumerate}
\end{definition}

\begin{observation}
If $\bm\in \fMtk$ and $\rho\in {}^{\ell_\bm}2$, then $\bm+\rho \in\fMtk$.  
\end{observation}

\begin{lemma}
 \label{litlem}
Let $0<\ell<\omega$ and let $\cB\subseteq {}^\ell 2$ be a linearly
independent set of vectors (in $({}^\ell2,+)$ over $(2,+_2,\cdot_2)$).
\begin{enumerate}
\item If $\cA\subseteq {}^\ell 2$, $|\cA|\geq 5$ and $\cA+\cA\subseteq \cB+\cB$, 
then for a unique $x\in {}^\ell 2$ we have $\cA+x\subseteq \cB$.  
\item Let $b^*\in \cB$. Suppose that  $\rho^0_i,\rho^1_i\in \big(\cB
  \cup (b^*+\cB)\big)\setminus\{{\mathbf 0},b^*\}$ (for $i<3$) are 
  such that   
  \begin{enumerate}
  \item[(a)] there are no repetitions in $\langle \rho^0_i,\rho^1_i:
    i<3\rangle$, and 
  \item[(b)] $\rho^0_i+\rho^1_i=\rho^0_j+\rho^1_j$ for $i<j<3$.
  \end{enumerate}
Then $\big\{\{\rho^0_i,\rho^1_i\}:i<3\}\subseteq \{\{b,b+b^*\}:b\in \cB,\
b\neq b^*\big\}$.
\end{enumerate}
\end{lemma}

\begin{proof}
Easy, for (1) see e.g.~\cite[Lemma 2.3]{RoRyxx}.
\end{proof}

\begin{theorem}
\label{522fortranslate}
Assume ${\rm NPr}_{\omega_1}(\lambda)$ and let $3\le \iota<\omega$. Then
there is a ccc forcing notion $\bbP$ of size $\lambda$ such that  
\[\begin{array}{l}
\forces_{\bbP}\mbox{`` for some $\Sigma^0_2$ $2\iota$--{\bf npots}--set
    $B\subseteq\can$ 
  there is a sequence
    }\langle\eta_\alpha:\alpha<\lambda\rangle\\
\qquad \mbox{ of distinct elements of $\can$ such that}\\
\qquad \big|(\eta_\alpha+B)\cap
(\eta_\beta+B)\big|\geq 2\iota\mbox{ for all }\alpha,\beta<\lambda\mbox{ ''.}
\end{array}\]
\end{theorem}

\begin{proof}
We may assume that $\lambda$ is uncountable. 

Fix a countable vocabulary $\tau=\{R_{n,\zeta}:n,\zeta< \omega\}$,  where
$R_{n,\zeta}$ is an $n$--ary relational symbol (for $n,\zeta<\omega$).  By
the assumption on $\lambda$, we may fix a model $\bbM=(\lambda,
\{R^\bbM_{n,\zeta}\}_{n,\zeta <\omega}) $ in the vocabulary $\tau$ with the 
universe $\lambda$ and an ordinal $\alpha^*<\omega_1$ such that: 
\begin{enumerate}
\item[$(\circledast)_{\rm a}$] for every $n$ and a quantifier free formula 
  $\varphi(x_0,\ldots,x_{n-1})\in \cL(\tau)$ there is $\zeta<\omega$ such
  that for all $a_0,\ldots, a_{n-1}\in \lambda$, 
\[\bbM\models\varphi[a_0,\ldots,a_{n-1}]\Leftrightarrow R_{n,\zeta}[a_0,\ldots,
a_{n-1}],\] 
\item[$(\circledast)_{\rm b}$] $\sup\{\rk(v,\bbM):\emptyset\neq v\in
  [\lambda]^{<\omega}\} <\alpha^*$, 
\item[$(\circledast)_{\rm c}$] the rank of every singleton is at least 0. 
\end{enumerate}
For a nonempty finite set $v\subseteq\lambda$ let $\rk(v)=\rk(v,\bbM)$, and 
let  $\zeta(v)<\omega$ and $k(v)<|v|$ be such that $R_{|v|,\zeta(v)},k(v)$ 
witness the rank of $v$. Thus letting $\{a_0,\ldots,a_k,\ldots a_{n-1}\}$ be the
increasing enumeration of $v$ and $k=k(v)$ and $\zeta= \zeta(v)$, we have    
\begin{enumerate}
\item[$(\circledast)_{\rm d}$] if $\rk(v)\geq 0$, then $\bbM\models
  R_{n,\zeta}[a_0,\ldots,a_k,\ldots,a_{n-1}]$ but there is no $a\in
  \lambda\setminus v$  such that  
\[\rk(v\cup\{a\})\geq \rk(v)\ \mbox{ and }\ \bbM\models R_{n,\zeta}
[a_0,\ldots,a_{k-1},a,a_{k+1},\ldots,a_{n-1}],\]  
\item[$(\circledast)_{\rm e}$] if $\rk(v)=-1$, then $\bbM\models R_{n,\zeta}  
  [a_0,\ldots,a_k,\ldots,a_{n-1}]$ but the set 
\[\big\{a\in\lambda:\bbM\models \varphi[a_0,\ldots,a_{k-1},a,a_{k+1},
  \ldots,a_{n-1}]\big\}\] 
is countable.  
\end{enumerate}
Without loss of generality we may also require that (for $\zeta=\zeta(v)$,
$n=|v|$) 
\begin{enumerate}
\item[$(\circledast)_{\rm f}$] for every $b_0,\ldots,b_{n-1}<\lambda$ 
\[\mbox{if }\ \bbM\models R_{n,\zeta}[b_0,\ldots,b_{n-1}] \mbox{ then }\
  b_0<\ldots <b_{n-1}.\]  
\end{enumerate}

\bigskip

Now we will define a forcing notion $\bbP$. {\em A condition\/} $p$ in
$\bbP$ is a tuple 
\[\big(w^p,n^p,M^p,\bar{\eta}^p,\bar{t}^p,\bar{r}^p,\bar{h}^p,\bar{g}^p,
\cM^p\big)= \big(w,n,M,\bar{\eta},\bar{t}, \bar{r},
\bar{h},\bar{g},\cM \big)\]     
such that the following demands $(*)_1$--$(*)_{11}$ are satisfied. 
\begin{enumerate}
\item[$(*)_1$] $w\in [\lambda]^{<\omega}$, $|w|\geq 5$, $0<n,M<\omega$. 
\item[$(*)_2$] $\bar{\eta}=\langle \eta_\alpha:\alpha\in w\rangle$ is a
  sequence of linearly independent vectors in ${}^n 2$ (over the field
  $\bbZ_2$); so in particular $\eta_\alpha\in {}^n2$ are pairwise
  distinct non-zero sequences (for $\alpha\in w$).  
\item[$(*)_3$] $\bar{t}= \langle t_m:m<M\rangle$, where $\emptyset\neq 
  t_m\subseteq {}^{n\geq} 2$ for $m<M$ is a tree in which all terminal
  branches are of length $n$ and $t_m\cap t_{m'}\cap {}^n 2=\emptyset$ for
  $m<m'<M$. 
\item[$(*)_4$] $\bar{r}=\langle r_m:m<M\rangle$, where $0<r_m\leq n$
  for $m<M$. 
\item[$(*)_5$] $\bar{h}=\langle h_i:i<\iota\rangle$, where $h_i:w^{\langle
    2\rangle}\longrightarrow M$. 
\item[$(*)_6$] $\bar{g} =\langle g_i:i<\iota\rangle$, where $g_i:w^{\langle
    2\rangle} \longrightarrow \bigcup\limits_{m<M} (t_m\cap {}^n 2)$, and 
$g_i(\alpha,\beta)\in t_{h_i(\alpha,\beta)}$ and  $\eta_\alpha+
  g_i(\alpha,\beta) =\eta_\beta+ g_i(\beta,\alpha)$ for  $(\alpha,\beta)\in
  w^{\langle 2\rangle}$ and $i<\iota$.  
\item[$(*)_7$] There are no repetitions in the sequence 
\[\langle g_i(\alpha,\beta): i<\iota,\ (\alpha,\beta)\in w^{\langle
  2\rangle}\rangle.\]  
\item[$(*)_8$] $\cM$ consists of all those $\bm\in\fMtk$ (see Definition
  \ref{fmtkDef}) that for some $\ell_*,w_*$ we  have  
\begin{enumerate}
\item[$(*)_8^{\rm a}$] $w_*\subseteq w$, $5\leq |w_*|$,
  $0<\ell_\bm=\ell_*\leq n$,   and for each $(\alpha,\beta)\in
  (w_*)^{\langle 2\rangle}$ and $i<\iota$ we have
  $r_{h_i(\alpha,\beta)}\leq \ell_*$,
\item[$(*)_8^{\rm b}$] $u_\bm=\{\eta_\alpha\rest \ell_*: \alpha\in
  w_*\}$ and $\eta_\alpha\rest\ell_*\neq \eta_\beta\rest\ell_*$ for distinct
  $\alpha,\beta\in w_*$,
\item[$(*)_8^{\rm c}$] $\bar{h}_\bm=\langle h^\bm_i:i<\iota\rangle$, where 
\[h^\bm_i:(u_\bm)^{\langle 2\rangle} \longrightarrow M:(\eta_\alpha\rest 
  \ell_*,\eta_\beta\rest \ell_*)\mapsto h_i(\alpha,\beta),\]
\item[$(*)_8^{\rm d}$] $\bar{g}_\bm=\langle g^\bm_i:i<\iota\rangle$, where 
\[g^\bm_i:(u_\bm)^{\langle 2\rangle} \longrightarrow \bigcup\limits_{m<M}
  (t_m\cap {}^{\ell_*} 2):(\eta_\alpha\rest \ell_*,\eta_\beta\rest 
  \ell_*)\mapsto g_i(\alpha,\beta)\rest \ell_*\]
\end{enumerate}
In the above situation we will write $\bm=\bm(\ell_*,w_*)=
\bm^p(\ell_*,w_*)$. (Note that $w_*$ is not determined uniquely by $\bm$ and
we may have $\bm(\ell,w_0)=\bm(\ell,w_1)$ for distinct $w_0,w_1\subseteq
w$. Also, the conditions $(*)^{\rm a}_8$--$(*)^{\rm d}_8$ alone do not
necessarily determine an element of $\fMtk$, but clearly for each
$w_*\subseteq w$ of size $\geq 5$ we have $\bm^p(n^p,w_*)\in \cM^p$.)   
\item[$(*)_9$] If $\bm(\ell,w_0), \bm(\ell,w_1)\in\cM$ , $\rho\in {}^\ell
  2$ and $\bm(\ell,w_0)\doteqdot\bm(\ell,w_1)+\rho$, then $\rk(w_0)=\rk(w_1)$,
  $\zeta(w_0)=\zeta(w_1)$, $k(w_0)=k(w_1)$ and if $\alpha\in w_0$, $\beta\in
  w_1$ are such that $|\alpha\cap w_0|=k(w_0) =k(w_1)=|\beta\cap w_1|$,
  then $(\eta_\alpha\rest \ell)+\rho=\eta_\beta\rest\ell$.  
\item[$(*)_{10}$] If $\bm(\ell_*,w_*)\in \cM$, $\alpha\in w_*$, $|\alpha\cap w_*|=
  k(w_*)$, $\rk(w_*)=-1$, and $\bm(\ell_*,w_*)\sqsubseteq^* \bn\in \cM$,
  then   $|\{\nu\in u_\bn:(\eta_\alpha\rest\ell_*)\trianglelefteq \nu\}|=1$.   
\item[$(*)_{11}$] If $\rho^0_i,\rho^1_i\in\bigcup\limits_{m<M} (t_m\cap
  {}^n2)$ (for $i<\iota$) are such that 
  \begin{enumerate}
  \item[(a)] there are no repetitions in $\langle \rho^0_i,\rho^1_i:
    i<\iota\rangle$, and 
  \item[(b)] $\rho^0_i+\rho^1_i=\rho^0_j+\rho^1_j$ for $i<j<\iota$,
  \end{enumerate}
then for some $\alpha,\beta\in w$ we have 
\[\big\{\{\rho^0_i,\rho^1_i\}:i<\iota\big\}=\big\{\{g_i(\alpha,\beta),
g_i(\beta,\alpha)\}: i<\iota\big\}.\]
\end{enumerate}
To define {\em the order $\leq$ of $\bbP$\/} we declare for $p,q\in \bbP$ that
$p\leq q$\quad if and only if 
\begin{itemize}
\item $w^p\subseteq w^q$, $n^p\leq n^q$, $M^p\leq M^q$, and 
\item $t^p_m=t^q_m\cap {}^{n^p\geq} 2$ and $r^p_m=r^q_m$ for all $m<M^p$, and    
\item $\eta^p_\alpha\trianglelefteq \eta^q_\alpha$ for all $\alpha\in w^p$,
  and    
\item $h^q_i\rest (w^p)^{\langle 2\rangle}= h^p_i$ and $g^p_i(\alpha,\beta)
  \trianglelefteq g^q_i(\alpha,\beta)$ for $i<\iota$ and $(\alpha,\beta)\in 
  (w^p)^{\langle 2\rangle}$.
\end{itemize}
\medskip

\begin{claim}
\label{cl7}
Assume $p=\big(w,n,M,\bar{\eta},\bar{t},\bar{h},\bar{g},\cM
\big)\in\bbP$. If $\bm\in \fMtk$ is such that $\ell_\bm=n$ and $|u_\bm|\geq
5$, then for some $\rho\in {}^n 2$ and $\bn\in\cM$ we have
$(\bm+\rho)\doteqdot \bn$.  
\end{claim}

\begin{proof}[Proof of the Claim]
Let $\bm\in \fMtk$ be such that $\ell_\bm=n$. It follows from Definition
\ref{mtkDef}(d) and clauses $(*)_6+(*)_{11}$ that 
\begin{enumerate}
\item[$(\boxdot)$] for every $(\nu,\eta)\in (u_\bm)^{\langle 2\rangle}$
  there is $(\alpha,\beta)\in w^{\langle 2\rangle}$ such that
  $\nu+\eta=\eta_\alpha +\eta_\beta$.
\end{enumerate}
By Lemma \ref{litlem} for some $\rho$ we have $u_\bm+\rho\subseteq
\{\eta_\alpha:\alpha\in w\}$. Let $w_0=\{\alpha\in w: \eta_\alpha+\rho \in
u_\bm\}$ and $\bn=\bm^p(n,w_0)\in\cM$. Using clause $(*)_{11}$ again we
easily conclude $(\bm+\rho)\doteqdot \bn$.  (Note that since $t_m\cap
t_{m'}\cap {}^n 2=\emptyset$ for $m<m'<M$, $h^\bm_i(\eta,\nu)$ is determined
by $g^\bm_i(\eta,\nu)$.) 
\end{proof}

\begin{claim}
\label{cl8}
\begin{enumerate}
\item $\bbP\neq\emptyset$ and $(\bbP,\leq)$ is a partial order.
\item For each $\beta<\lambda$ and $n_0,M_0<\omega$ the set 
\[D_\beta^{n_0,M_0}=\big \{p\in\bbP: n^p>n_0\ \wedge\ M^p>M_0\ \wedge\
\beta\in w^p \big\}\]
is open dense in $\bbP$. 
\end{enumerate}
\end{claim}

\begin{proof}[Proof of the Claim]
(1)\quad Should be clear. 
\medskip

\noindent (2)\quad   Let $p\in\bbP$, $\beta\in\lambda\setminus w^p$. Put
$N=|w^p|\cdot \iota+2$.  

We will define a condition $q\in \bbP$ such that $q\geq p$ and 
\[w^q=w^p\cup\{\beta\},\quad n^q=n^p+N>n^p+1,\quad 
 M^q=M^p+N-2> M^p+1.\]
For $\alpha\in w^p$ we set $\eta^q_\alpha=\eta^p_\alpha\conc \langle
\underbrace{0, \ldots, 0}_{N}\rangle $ and we also let $\eta^q_\beta=\langle
\underbrace{0,\ldots, 0}_{n^p+1}\rangle \conc \langle
\underbrace{1,\ldots,1}_{N-1} \rangle$.  Next, if $(\alpha_0,\alpha_1)\in
(w^p)^{\langle 2\rangle}$, then for all $i<\iota$
\[h^q_i(\alpha_0,\alpha_1)=h^p_i(\alpha_0,\alpha_1)\quad \mbox{ and } \quad  
g^q_i(\alpha_0,\alpha_1)=g^p_i(\alpha_0,\alpha_1)\conc \langle
\underbrace{0, \ldots, 0}_{N}\rangle.\]
If $\alpha\in w^p$ and $j=|w^p\cap \alpha|$, then for  $i<\iota$:
\begin{itemize}
\item $g^q_i(\alpha,\beta)=\langle \underbrace{0,\ldots,0}_{n^p}\rangle
  \conc\langle 1\rangle \conc\langle\underbrace{0,\ldots,0}_{j\iota+i+1}\rangle 
  \conc\langle\underbrace{1,\ldots,1}_{N-j\iota-i-2} \rangle$,  
\item $g^q_i(\beta,\alpha)=\eta^p_\alpha\conc\langle \underbrace{1, \ldots,
    1}_{j\iota+i+2}\rangle \conc \langle \underbrace{0,
    \ldots,0}_{N-j\iota-i-2} \rangle$,  
\item $h^q_i(\beta,\alpha)=h^q_i(\alpha,\beta)=M^p+j\iota+i$. 
\end{itemize}
We also set:
\begin{itemize}
\item if $m<M^p$, then $r^q_m=r^p_m$ and 
\[t^q_m=\{\eta\in {}^{n^q\geq} 2: \eta\rest n^p\in t^p_m\ \wedge\ (\forall
j<n^q)( n^p\leq j<|\eta| \Rightarrow  \eta(j)=0)\}\]
and 
\item if $M^p\leq m<M^q$, $m=M^p+j\iota+i$, $i<\iota$ and 
  $j<|w^p|$, then  $r^q_m=n^q$ and 
\[t^q_m=\{g^q_i(\alpha,\beta)\rest \ell, g^q_i(\beta,\alpha)\rest \ell:
\ell\leq n^q\},\] 
where $\alpha\in w^p$ is such that $|\alpha\cap w^p|=j$. 
\end{itemize}
Now letting $\cM^q$ be defined as in $(*)_8$ we check that
\[q=\big(w^q,n^q,M^q,\bar{\eta}^q,\bar{t}^q,\bar{r}^q, \bar{h}^q,\bar{g}^q,  
\cM^p\big)\in \bbP.\] 
Demands $(*)_1$--$(*)_8$ are pretty straightforward.
\bigskip

\noindent {\bf RE $(*)_9$\ :}\quad To justify clause $(*)_9$, suppose that
$\bm^q(\ell,w_0), \bm^q(\ell,w_1)\in  \cM^q$, $\rho\in {}^\ell 2$ and
$\bm^q(\ell,w_0)\doteqdot \bm^q(\ell,w_1)+\rho$, and consider the following
two cases. 

\noindent {\sc Case 1:}\quad $\beta\notin w_0\cup w_1$\\
Then letting $\ell^*=\min(\ell,n^p)$ and $\rho^*=\rho\rest \ell^*$ we see
that  $\bm^p(\ell^*,w_0)\doteqdot\bm^p(\ell^*,w_1)+\rho^*$ (and both belong
to $\cM^p$). Hence clause $(*)_9$ for $p$ applies. 

\noindent {\sc Case 2:}\quad $\beta\in w_0\cup w_1$\\
Say, $\beta\in w_0$. If $\alpha\in w_0\setminus \{\beta\}$, then
$h^q_i(\alpha,\beta) =h^q_i(\beta,\alpha)\geq M^p$ and $r^q_{h^q_i(\alpha,
  \beta}=n^q$. Consequently, $\ell=n^q$. Moreover,  
\[(\gamma,\delta)\in (w^q)^{\langle 2\rangle}\ \wedge\ h^q_j(\gamma,
\delta)= h^q_i(\alpha,\beta)\quad \Rightarrow \quad
\{\gamma,\delta\}=\{\alpha,\beta\}.\] 
Therefore, $\beta\in w_1$ and $w_1=w_0$ and since $|w_1|\geq 5$, the linear
independence of $\bar{\eta}$ implies $\rho={\mathbf 0}$.
\bigskip

\noindent {\bf RE $(*)_{10}$\ :}\quad Concerning clause $(*)_{10}$, suppose
that $\bm^q(\ell_0,w_0),\bm^q(\ell_1,w_1) \in\cM^q$, $\alpha\in w_0$,
$|\alpha\cap w_0|= k(w_0)$, $\rk(w_0)=-1$, and
$\bm^q(\ell_0,w_0)\sqsubseteq^* \bm^q(\ell_1,w_1)$. Assume towards 
contradiction that there are $\alpha_0,\alpha_1\in  w_1$ such that
\[\eta^q_{\alpha_0}\rest \ell_1\neq \eta^q_{\alpha_1}\rest \ell_1\ \wedge \  
\eta^q_\alpha\rest\ell_0\vtl \eta^q_{\alpha_0}\ \wedge\ 
\eta^q_\alpha\rest\ell_0\vtl \eta^q_{\alpha_1}.\] 
Suppose $\beta\in w_0\cup w_1$. Then looking at the function $h^q_i$ in a
manner similar to considerations for clause $(*)_9$ we get $\beta\in w_0\cap
w_1$. Let $\beta'\in w_0\setminus \{\beta\}$. Then $h_0^q(\beta,\beta')\geq
M^p$ and hence $r_{h_0(\beta,\beta')}^q=n^q=\ell_0=\ell_1$, contradicting
our assumptions. Therefore $\beta\notin w_0\cup w_1$. But then we
immediately get contradiction with clause $(*)_{10}$ for $p$. 
\bigskip

\noindent {\bf RE $(*)_{11}$\ :}\quad Let us argue that $(*)_{11}$ is
satisfied as well and for this suppose that $\rho^0_i,\rho^1_i\in
\bigcup\limits_{m<M^q} (t_m\cap {}^{n^q}2)$ (for $i<\iota$) are such that   
  \begin{enumerate}
  \item[(a)] there are no repetitions in $\langle \rho^0_i,\rho^1_i:
    i<\iota\rangle$, and 
  \item[(b)] $\rho^0_i+\rho^1_i=\rho^0_j+\rho^1_j$ for $i<j<\iota$.
  \end{enumerate}
Clearly, if 
\begin{enumerate}
\item[$(\odot)_1$] all $\rho^0_i,\rho^1_i$ are from $\bigcup\limits_{m<M^p}
  t_m$,  
\end{enumerate}
then we may use the condition $(*)_{11}$ for $p$ and conclude that for some
$\alpha_0,\alpha_1\in w^p$ we have  
\[\big\{\{\rho^0_i,\rho^1_i\}:i<\iota\big\}=\big\{\{g_i(\alpha_0, \alpha_1),  
g_i(\alpha_1,\alpha_0)\}: i<\iota\big\}.\]
Now note that if $\rho_0,\rho_1,\rho_2,\rho_3\in \bigcup\limits_{m<M^q}
(t_m\cap {}^{n^q}2)$, $\rho_0+\rho_1=\rho_2+\rho_3$ and $\rho_0\in
\bigcup\limits_{m<M^p} (t_m\cap {}^{n^q}2)$ but $\rho_1\notin
\bigcup\limits_{m<M^p} (t_m\cap {}^{n^q}2)$, then $\{\rho_0,\rho_1\}=
\{\rho_2,\rho_3\}$. Hence easily, if $(\odot)_1$ fails we must have
\begin{enumerate}
\item[$(\odot)_2$] $\rho_i^0,\rho_i^1\in \bigcup\limits_{m=M^p}^{M^q-1}
  (t_m\cap {}^{n^q}2)$  for $i<\iota$.
\end{enumerate}
But then  necessarily 
\[\big\{\{\rho^0_i\rest [n^p,n^q), \rho^1_i\rest [n^p,n^q)\}: i<\iota \big\}
  \subseteq \big\{\{g_i(\alpha, \beta)\rest [n^p,n^q),
  g_i(\beta,\alpha)\rest [n^p,n^q)\}: i<\iota,\ \alpha\in w^p\big\}.\]
(Use Lemma \ref{litlem}(2), remember $\iota\geq 3$.) Since
$\big(g_i(\alpha,\beta)+ g_i(\beta,\alpha)\big)\rest n^p= \eta^p_\alpha$ we
easily conclude that for some $\alpha\in w^p$ we have 
\[\big\{\{\rho^0_i,\rho^1_i\}:i<\iota\big\}=\big\{\{g_i(\alpha, \beta),  
g_i(\beta,\alpha)\}: i<\iota\big\}.\]
\medskip

Finally, it should be clear that $q$ is a condition stronger than $p$. 
\end{proof}

\begin{claim}
  \label{cl9}
The forcing notion $\bbP$ has the Knaster property.
\end{claim}

\begin{proof}[Proof of the Claim]
Suppose that $\langle p_\xi:\xi<\omega_1\rangle$ is a sequence of pairwise 
distinct conditions from $\bbP$ and let
\[p_\xi=\big(w_\xi,n_\xi,M_\xi,\bar{\eta}_\xi, \bar{t}_\xi, \bar{r}_\xi, 
\bar{h}_\xi,\bar{g}_\xi, \cM_\xi\big)\]
where $\bar{\eta}_\xi=\langle\eta^\xi_\alpha: \alpha\in w_\xi\rangle$,
$\bar{t}_\xi=\langle t^\xi_m:m<M_\xi\rangle$, $\bar{r}_\xi=\langle
r^\xi_m:m<M_\xi\rangle$,  and $\bar{h}_\xi= \langle
h^\xi_i:i<\iota\rangle$, $\bar{g}_\xi=\langle g^\xi_i:i<\iota\rangle$.  By a
standard $\Delta$--system cleaning procedure we may find an uncountable set
$A\subseteq \omega_1$ such that the following demands $(*)_{12}$--$(*)_{15}$
are satisfied. 
\begin{enumerate}
\item[$(*)_{12}$] $\{w_\xi:\xi\in A\}$ forms a $\Delta$--system.
\item[$(*)_{13}$] If  $\xi,\varsigma\in A$, then $|w_\xi|=|w_{\varsigma}|$ ,
  $n_\xi=n_{\varsigma}$, $M_\xi=M_{\varsigma}$,  and
  $t^\xi_m=t^{\varsigma}_m$ and $r^\xi_m=r^\varsigma_m$ (for  $m<M_\xi$).      
\item[$(*)_{14}$] If $\xi<\varsigma$ are from $A$ and
  $\pi:w_\xi\longrightarrow w_{\varsigma}$ is the order isomorphism, then 
  \begin{enumerate}
  \item[(a)] $\pi(\alpha)=\alpha$ for $\alpha\in w_\xi\cap w_{\varsigma}$, 
  \item[(b)] if $\emptyset\neq v\subseteq w_\xi$, then $\rk(v)=\rk(\pi[v])$,
    $\zeta(v)= \zeta(\pi[v])$ and $k(v)=k(\pi[v])$, 
  \item[(c)] $\eta_\alpha^\xi=\eta_{\pi(\alpha)}^{\varsigma}$ (for
    $\alpha\in  w_\xi$), 
  \item[(d)] $g_i(\alpha,\beta)=g_i(\pi(\alpha),\pi(\beta))$ and
    $h_i(\alpha,\beta) =h_i(\pi(\alpha),\pi(\beta))$ for $(\alpha,\beta)\in
    (w_\xi)^{\langle 2\rangle}$ and $i<\iota$,
  \end{enumerate}
and 
\item[$(*)_{15}$] $\cM_\xi=\cM_\varsigma$ (this actually follows from the
previous demands). 
\end{enumerate}
Following the pattern of Claim \ref{cl8}(2) we will argue that for distinct
$\xi,\varsigma$ from $A$ the conditions $p_\xi,p_\varsigma$ are
compatible. So let $\xi,\varsigma\in A$, $\xi<\varsigma$ and let
$\pi:w_\xi\longrightarrow w_\varsigma$ be the order isomorphism.   
We will define $q=\big(w,n,M,\bar{\eta}, \bar{t}, \bar{r}, \bar{h},\bar{g}, 
\cM\big)$ where $\bar{\eta}=\langle\eta_\alpha: \alpha\in w\rangle$,
$\bar{t}=\langle t_m:m<M\rangle$, $\bar{r}=\langle r_m:m<M\rangle$, and
$\bar{h}= \langle h_i:i<\iota \rangle$, $\bar{g}=\langle g_i:i<\iota\rangle$.  

Let $w_\xi\cap w_\varsigma=\{\alpha_0,\ldots,\alpha_{k-1}\}$,
$w_\xi\setminus w_\varsigma=\{\beta_0,\ldots,\beta_{\ell-1}\}$ and
$w_\varsigma\setminus w_\xi=\{\gamma_0,\ldots,\gamma_{\ell-1}\}$ be
the increasing enumerations. 

We set $N_0=\iota\cdot \ell (\ell+k)+\iota\cdot\frac{\ell(\ell-1)}{2}+1$,
$N=N_0+\ell+1$, and we define  
\begin{enumerate}
\item[$(*)_{16}$] $w=w_\xi\cup w_\varsigma$, $n=n_\xi+N$, and
  $M=M_\xi+1$; 
\item[$(*)_{17}$] $\eta_\alpha=\eta_\alpha\conc \langle
  \underbrace{0, \ldots,0}_{N}\rangle $ for $\alpha\in w_\xi$ and we
  also let  for $c<\ell$
\[\eta_{\gamma_c} =\eta_{\gamma_c}^\xi\conc \langle 0\rangle \conc \langle  
  \underbrace{1,\ldots,1}_{N_0} \rangle \conc \langle  
  \underbrace{0,\ldots,0}_c \rangle \conc \langle  
  \underbrace{1,\ldots,1}_{\ell-c} \rangle.\]
\end{enumerate}
Next we are going to define $h_i(\alpha,\beta)$ and
$g_i(\alpha,\beta)$ for $(\alpha,\beta)\in w^{\langle
  2\rangle}$. For $d<N_0$ let  
\[\nu_d=\langle \underbrace{0,\ldots,0}_d\rangle \conc \langle
  1\rangle\conc \langle \underbrace{0,\ldots,0}_{N_0-d-1}\rangle \in 
  {}^{N_0} 2,\quad\mbox{ and }\quad \nu^*_d={\mathbf 1}+\nu_d \in
  {}^{N_0} 2\]  
and note that $\{\nu_d:d<N_0-1\}\cup\{ {\mathbf 1}\}$ are linearly
independent in ${}^{N_0} 2$. Fix a bijection
\[\Theta:(k\times\ell\times\iota\times\{0\}) \cup(\{(a,b)\in
\ell^2:a<b\}\times \iota\times\{1\})\cup  (\ell\times\ell\times 
  \iota\times\{2\})\longrightarrow N_0-1\]  
and define $h_i,g_i$ as follows. 
\begin{enumerate}
\item[$(*)_{18}^{\rm a}$]  If $(\alpha,\beta)\in (w_\xi)^{\langle
    2\rangle}$ and  $i<\iota$, then 

$h_i(\alpha,\beta)=h^\xi_i(\alpha,\beta)$ and $g_i(\alpha,\beta) =
g^\xi_i(\alpha,\beta)\conc \langle \underbrace{0, \ldots, 0}_{N}\rangle$. 
\item[$(*)_{18}^{\rm b}$] If $a<k$, $c<\ell$  and $i<\iota$, then
  $h_i(\alpha_a,\gamma_c)= h^\varsigma_i (\alpha_a,\gamma_c)$ and
  $h_i(\gamma_c, \alpha_a)= h^\varsigma_i (\gamma_c, \alpha_a)$, and
\[\begin{array}{l}
g_i(\alpha_a,\gamma_c) =g^\varsigma_i(\alpha_a,\gamma_c)\conc \langle
1\rangle \conc \nu_{\Theta(a,c,i,0)}\conc \langle  
  \underbrace{0,\ldots,0}_\ell \rangle\quad \mbox{ and }\\
 g_i(\gamma_c, \alpha_a) =g^\varsigma_i(\gamma_c,\alpha_a)\conc
    \langle  1\rangle \conc \nu^*_{\Theta(a,c,i,0)}\conc \langle  
  \underbrace{0,\ldots,0}_c\rangle \conc \langle  
  \underbrace{1,\ldots,1}_{\ell-c} \rangle.
\end{array}\]
\item[$(*)_{18}^{\rm c}$] If $b<c<\ell$  and $i<\iota$, then 
  $h_i(\gamma_b,\gamma_c)=h^\varsigma_i(\gamma_b,\gamma_c)$, $h_i(\gamma_c,
  \gamma_b)= h^\varsigma_i(\gamma_c,\gamma_b)$, and 
\[\begin{array}{l}
g_i(\gamma_b,\gamma_c) =g ^\varsigma_i(\gamma_b,\gamma_c)\conc \langle 
1\rangle \conc \nu_{\Theta(b,c,i,1)}\conc \langle  
  \underbrace{0,\ldots,0}_b\rangle \conc \langle  
  \underbrace{1,\ldots,1}_{\ell-b} \rangle\quad \mbox{ and }\\
 g_i(\gamma_c, \gamma_b) =g^\varsigma_i(\gamma_c,\gamma_b)\conc  
    \langle  1\rangle \conc \nu_{\Theta(b,c,i,1)}\conc \langle  
  \underbrace{0,\ldots,0}_c\rangle \conc \langle  
  \underbrace{1,\ldots,1}_{\ell-c} \rangle
\end{array}\]
(note: $\nu_\Theta$ not $\nu^*_\Theta$).
\item[$(*)_{18}^{\rm d}$] If $b<\ell$, $c<\ell$  and $b\neq c$ and
  $i<\iota$, then  $h_i(\beta_b,\gamma_c)=h_i(\gamma_c, \beta_b)=
  M_\xi=M_\varsigma$,  and 
\[\begin{array}{l}
g_i(\beta_b,\gamma_c) =g ^\xi_i(\beta_b,\beta_c)\conc \langle 
1\rangle \conc \nu_{\Theta(b,c,i,2)}\conc \langle  
  \underbrace{0,\ldots,0}_c\rangle \conc \langle  
  \underbrace{1,\ldots,1}_{\ell-c} \rangle\quad \mbox{ and }\\
 g_i(\gamma_c, \beta_b) =g^\varsigma_i(\gamma_c,\gamma_b) \conc 
    \langle  1\rangle \conc \nu^*_{\Theta(b,c,i,2)}\conc \langle  
  \underbrace{0,\ldots,0}_\ell\rangle.
\end{array}\]
\item[$(*)_{18}^{\rm e}$] If $b<\ell$ and $i<\iota$, then
  $h_i(\beta_b,\gamma_b)=h_i(\gamma_b, \beta_b)= M_\xi=M_\varsigma$,  and  
\[\begin{array}{l}
g_i(\beta_b,\gamma_b) =\eta ^\xi_{\beta_b}\conc \langle 
1\rangle \conc \nu_{\Theta(b,b,i,2)}\conc \langle  
  \underbrace{0,\ldots,0}_b\rangle \conc \langle  
  \underbrace{1,\ldots,1}_{\ell-b} \rangle\quad \mbox{ and }\\
 g_i(\gamma_b, \beta_b) =\eta^\varsigma_{\gamma_b}\conc 
    \langle  1\rangle \conc \nu^*_{\Theta(b,b,i,2)}\conc \langle  
  \underbrace{0,\ldots,0}_\ell\rangle.
\end{array}\]
\end{enumerate}
We also set:
\begin{enumerate}
\item[$(*)_{19}$] $r_m=r^\xi_m$ for $m<M_\xi$, $r_{M_\xi}= n$ and if
  $m<M_\xi$, then 
\[\begin{array}{ll}
t_m=&\big\{\eta\in {}^{n\geq} 2: \eta\rest n_\xi\in t^\xi_m\ \wedge\ (\forall
j<n)( n\leq j<|\eta| \Rightarrow  \eta(j)=0)\big\}\ \cup\ \\
& \big\{g_i(\delta,\vare)\rest n': (\delta,\vare)\in w^{\langle
  2\rangle},\  i<\iota, \mbox{ and }n'\leq n\mbox{ and } h_i(\delta,\vare)=
  m\big\} 
\end{array}\]
and 
\[t_{M_\xi}=\big\{g_i(\delta,\vare)\rest n': (\delta,\vare)\in w^{\langle 
  2\rangle},\  i<\iota, \mbox{ and }n'\leq n\mbox{ and } h_i(\delta,\vare)=
M_\xi\big\}. \]
\end{enumerate}
Now letting $\cM$ be defined by $(*)_8$ we claim that
\[q=\big(w,n,M,\bar{\eta},\bar{t},\bar{r},\bar{h},\bar{g},\cM\big)\in \bbP.\]  
Demands $(*)_1$--$(*)_8$ are pretty straightforward.
\bigskip

\noindent {\bf RE $(*)_9$\ :}\quad To justify clause $(*)_9$, suppose that
$\bm(\ell,w'), \bm(\ell,w'')\in \cM$, $\rho\in {}^\ell 2$ and
$\bm(\ell,w')\doteqdot \bm(\ell,w'')+\rho$, and consider the following three
cases. 
\medskip

\noindent {\sc Case 1:}\quad $w'\subseteq w_\xi$\\
Then for each $(\delta,\vare)\in (w')^{\langle 2\rangle}$ we have
$h_i(\delta,\vare)<M_\xi$, so this also holds for $(\delta,\vare)\in
(w'')^{\langle 2\rangle}$. Consequently, either $w''\subseteq w_\xi$ or
$w''\subseteq w_\varsigma$. 

If $w''\subseteq w_\xi$, then let $\ell'=\min(\ell,n_\xi)$ and consider
$\bm^{p_\xi}(w',\ell'), \bm^{p_\xi}(w'',\ell')\in \cM_\xi$. Using clause
$(*)_9$ for $p_\xi$ we immediately obtain the desired conclusion. 

If $w''\subseteq w_\varsigma$, then we let $\ell'=\min(\ell,n_\xi)$ and 
we consider $\bm^{p_\xi}(w',\ell')$ and $\bm^{p_\xi}(\pi^{-1}[w''],\ell')$
(both from $\cM_\xi$). By $(*)_{14}$, clause $(*)_9$ for $p_\xi$ applies to
them and we get 
\begin{itemize}
\item $\rk(w')=\rk(\pi^{-1}[w''])$, $\zeta(w')=\zeta(\pi^{-1}[w''])$,
$k(w')=k(\pi^{-1}[w''])$ and 
\item if $\delta\in w'$, $\vare\in \pi^{-1}[w'']$ are such that $|\delta\cap
  w'|=k(w') =k(\pi^{-1}[w''])=|\vare\cap\pi^{-1}[w'']|$,  then
  $(\eta_\delta^{p_\xi}\rest \ell')+\rho=\eta_\vare^{p_\xi} \rest\ell'$.  
\end{itemize}
By $(*)_{14}$ this immediately implies the desired conclusion. 
\medskip

\noindent {\sc Case 2:}\quad $w'\subseteq w_\varsigma$\\
Same as the previous case, just interchanging $\xi$ and $\varsigma$. 
\medskip

\noindent {\sc Case 3:}\quad $w'\setminus w_\xi\neq \emptyset \neq
w'\setminus w_\varsigma$\\
Then for some $(\delta,\vare)\in (w')^{\langle 2\rangle}$ we have
$h_i(\delta,\vare)=M_\xi$, so necessarily $\ell=r_{M_\xi}=n$. Hence
$\{\eta_\alpha:\alpha\in w'\}=\{\eta_\alpha+\rho:\alpha\in w''\}$ and since
$|w'|\geq 5$, the linear independence of $\bar{\eta}$ implies $\rho={\mathbf
  0}$ and $w'=w''$ and the desired conclusion follows. 
\bigskip

\noindent {\bf RE $(*)_{10}$\ :}\quad Let us prove clause $(*)_{10}$
now. Suppose that $\bm(\ell_0,w'),\bm(\ell_1,w'') \in\cM$, $\delta\in w'$,
$|\delta\cap w'|= k(w')$, $\rk(w')=-1$, and  $\bm(\ell_0,w')\sqsubseteq^* 
\bm(\ell_1,w'')$. Assume towards contradiction that there are
$\vare_0,\vare_1\in  w''$ such that 
\begin{enumerate}
\item[$(\otimes)_0$] $\eta_{\vare_0}\rest \ell_1\neq \eta_{\vare_1}\rest
  \ell_1$ and $\eta_\delta\rest\ell_0\vtl \eta_{\vare_0}$ and
  $\eta_\delta\rest\ell_0\vtl \eta_{\vare_1}$. 
\end{enumerate}
Without loss of generality $|w''|=|w'|+1\geq 6$.

Since we must have $\ell_0<n$, for no $\alpha,\beta\in w'$ we can have
$h_i(\alpha,\beta)=M_\xi$. Therefore either $w'\subseteq w_\xi$ or
$w'\subseteq w_\varsigma$. Also, 
\begin{enumerate}
\item[$(\otimes)_1$] if $(\alpha,\beta)\in (w'')^{\langle 
  2\rangle}\setminus \{(\vare_0,\vare_1),(\vare_1, \vare_0)\}$ then 
$h_i(\alpha,\beta)<M_\xi$ for $i<\iota$.  
\end{enumerate}
Note that 
\begin{enumerate}
\item[$(\otimes)_2$] if $(\alpha,\beta)\in (w_\xi)^{\langle 2\rangle} \cup 
(w_\varsigma)^{\langle 2\rangle}$ then $\min(\{\ell:\eta_\alpha (\ell) \neq 
\eta_\beta(\ell)\})<n_\xi$ and there are no repetitions in the sequence
$\langle g_i(\alpha,\beta)\rest n_\xi,g_i(\beta,\alpha)\rest n_\xi:
i<\iota\rangle$.   
\end{enumerate}
Let $\ell^*=\min(\ell_1,n_\xi)$. 

Now, if $w'\cup w''\subseteq w_\xi$, then considering
$\bm(\ell_0,w')$ and $\bm(\ell^*,w'')$ (and remembering
$(\otimes)_2$) we see that $\ell_0<n_\xi$,
$\bm^{p_\xi}(\ell_0,w')\sqsubseteq^*\bm^{p_\xi}(\ell^*,w'')$ and they
have the property contradicting $(*)_{10}$ for $p_\xi$.   

If $w'\cup w''\subseteq w_\varsigma$, then in a similar manner we get
contradiction with $(*)_{10}$ for $p_\varsigma$. 

If $w'\subseteq w_\xi$ and $w''\subseteq w_\varsigma$ then one easily
verifies that $\ell_0<n_\xi$ and
$\bm^{p_\xi}(\ell_0,w')\sqsubseteq^*
\bm^{p_\xi}(\ell^*,\pi^{-1}[w''])$  provide a counterexample for
$(*)_{10}$ for $p_\xi$. Similarly if $w'\subseteq w_\varsigma$ and
$w''\subseteq w_\xi$.  

Consequently, the only possibility left is that $w''\setminus w_\xi\neq
\emptyset \neq w''\setminus w_\varsigma$ and it follows from
$(\otimes)_1$ that $|w''\setminus w_\xi| = |w''\setminus
w_\varsigma|=1$. Let $\{\beta_b\}=w''\setminus w_\varsigma$ and
$\{\gamma_c\}= w''\setminus w_\xi$; then $\{\vare_0,\vare_1\}=
\{\beta_b,\gamma_c\}$.  

Assume $w'\subseteq w_\xi$ (the case when $w'\subseteq w_\varsigma$ can be
handled similarly). If we had $b\neq c$, then
$\eta_{\beta_b}\rest n_\xi=\eta_{\beta_b}^{p_\xi} \rest n_\xi \neq
\eta_{\gamma_c}^{p_\varsigma} \rest n_\xi= \eta_{\gamma_c} \rest
n_\xi$. Since $w''\subseteq (w_\xi\cap  w_\varsigma)\cup
\{\beta_b,\gamma_c\}$ we could see that 
$\ell_0<n_\xi$ and $\bm^{p_\xi}(\ell_0,w') \sqsubseteq^*
\bm^{p_\xi}(\ell^*,\pi^{-1}[w''])$ would provide a  counterexample for
$(*)_{10}$ for $p_\xi$. Consequently, $b=c$ and $\ell_1>n_\xi$. Now,
remembering $(\otimes)_0$, $\eta^{p_\xi}_\delta\rest
\ell_0=\eta^{p_\xi}_{\beta_b}\rest \ell_0$ and
$\bm^{p_\xi}(\ell_0,w')\doteqdot \bm^{p_\xi}(\ell_0,w'' \setminus 
\{\gamma_b\})$, so by $(*)_9$ for $p_\xi$ we conclude
\[\rk(w''\setminus\{\gamma_b\})=-1\quad 
\mbox{ and }\quad |\beta_b\cap (w''\setminus\{\gamma_b\})|=
k(w''\setminus\{\gamma_b\}).\]
Let $\zeta^*=\zeta(w''\setminus \{\gamma_b\})$ and
$k^*=k(w''\setminus\{\gamma_b\})$. For $\vare\in A\setminus\{\xi\}$ let
$\pi^\vare:w_\xi\longrightarrow w_\vare$ be the order isomorphism and let
$\gamma(\vare)\in \pi^\vare\big[w''\setminus\{\gamma_b\}\big]$ be such that
$|\pi^\vare\big[w''\setminus\{\gamma_b\}\big]\cap\gamma(\vare)|= k^*$ 
(necessarily $\gamma(\vare)= \pi^\vare(\beta_b) \in w_\vare\setminus
w_\xi$). Then  
\begin{itemize}
\item $\pi^\vare[w''\setminus\{\gamma_b\}]=\big(w''\cap (w_\xi\cap 
w_\vare)\big)\cup\{ \gamma(\vare)\}=w''\setminus\{\beta_b,\gamma_b\}
\cup\{\gamma(\vare)\}$, 
\item $\rk\Big(\pi^\vare[w''\setminus\{\gamma_b\}]\Big)=-1$, and 
$\zeta\Big(\pi^\vare[w''\setminus\{\gamma_b\}]\Big)=\zeta^*$, and 
\item $k\Big(\pi^\vare[w''\setminus\{\gamma_b\}]\Big)=k^*=|\pi^\vare [w''
  \setminus\{\gamma_b\}] \cap \gamma(\vare)|$. 
\end{itemize}
Hence $\bbM\models R_{|w'|,\zeta^*}\big [w''\setminus\{\beta_b,\gamma_b\}
\cup \{\gamma(\vare)\}\big]$ for each $\vare\in A\setminus
\{\xi\}$. Consequently, the set 
\[\Big\{\alpha<\lambda:\bbM\models  R_{|w'|,\zeta^*}\big [w''\setminus\{\beta_b,\gamma_b\}
\cup \{\alpha\}\big]\Big\}\]
is uncountable, contradicting $(\circledast)_{\rm e}$. 
\bigskip

\noindent {\bf RE $(*)_{11}$\ :}\quad Let us argue that $(*)_{11}$ is
satisfied as well and for this suppose that $\rho^0_i,\rho^1_i\in
\bigcup\limits_{m<M} (t_m\cap {}^n2)$ (for $i<\iota$) are such that   
  \begin{enumerate}
  \item[(a)] there are no repetitions in $\langle \rho^0_i,\rho^1_i:
    i<\iota\rangle$, and 
  \item[(b)] $\rho^0_i+\rho^1_i=\rho^0_j+\rho^1_j$ for $i<j<\iota$.
  \end{enumerate}
Clearly, if all $\rho^0_i,\rho^1_i$ are form $\rho\conc\langle
\underbrace{0,\ldots, 0}_N\rangle$, then we may use condition $(*)_{11}$
for $p_\xi$ and conclude that for some $\alpha_0,\alpha_1\in w_\xi$ we have   
\[\big\{\{\rho^0_i,\rho^1_i\}:i<\iota\big\}=\big\{\{g_i(\alpha_0, \alpha_1),  
g_i(\alpha_1,\alpha_0)\}: i<\iota\big\}.\]
So assume that we are not in the situation when all  $\rho^0_i,\rho^1_i$ are
form $\rho\conc\langle \underbrace{0,\ldots, 0}_N\rangle$.

Note that if $\rho\in \bigcup\limits_{m<M} (t_m\cap {}^n 2)$ and
$\rho(n_\xi)=0$, then $\rho\rest [n_\xi,n)={\mathbf 0}$. Hence, remembering 
definitions in $(*)_{18}$, if $\rho_0,\rho_1,\rho_2,\rho_3\in
\bigcup\limits_{m<M} (t_m\cap {}^n 2)$, $\rho_0+\rho_1=\rho_2+\rho_3$ and
$\rho_0(n_\xi)=0$ but  $\rho_1(n_\xi)=1$, then $\{\rho_0,\rho_1\}=
\{\rho_2,\rho_3\}$. Therefore, under current assumption, we must have
$\rho^0_i(n_\xi)=\rho^1_i(n_\xi)=1$ for all $i<\iota$. Define 

$B=\{(\alpha_a,\gamma_c):a<k\ \&\ c<\ell\}$,

$C=\{(\gamma_b,\gamma_c):b<c<\ell\}$,

$D=\{(\beta_b,\gamma_c): b<\ell\ \&\ c<\ell\ \&\ b\neq c\}$, 

$E=\{(\beta_b,\gamma_b): b<\ell\}$. 

\noindent (These four sets correspond to the conditions $(*)_{18}^{\rm 
  b}$--$(*)_{18}^{\rm e}$ in the definition of $g_i$.) Clearly,
$\rho^0_i(n_\xi)=\rho^1_i(n_\xi)=1$ implies that 
\[\rho^0_i,\rho^1_i\in
\{g_j(\vare_0,\vare_1), g_j(\vare_1,\vare_0):(\vare_0,\vare_1)\in B\cup
C\cup D\cup E,\ j<\iota\}.\]  
Note also that for each $d<N_0-1$,
\begin{enumerate}
\item[$(\boxtimes)_a$] the set $\{\rho\in \bigcup\limits_{m<M} (t_m\cap {}^n
  2): \rho \rest \big(n_\xi,n_\xi+N_0\big] =\nu_d\}$ is not empty but it has
  at most two elements, and   
\item[$(\boxtimes)_b$] $|\{\rho\in \bigcup\limits_{m<M} (t_m\cap {}^n 2): \rho \rest
    \big(n_\xi,n_\xi+N_0\big]=\nu_d\}|=2$ if and only if
    $d=\Theta(b,c,i,1)$ for some $b<c<\ell$ and $i<\iota$, and
\item[$(\boxtimes)_c$] the set $\{\rho\in \bigcup\limits_{m<M} (t_m\cap {}^n
  2): \rho \rest \big(n_\xi,n_\xi+N_0\big] =\nu^*_d\}$ has at most one
  element, and  
\item[$(\boxtimes)_d$] $\{\rho\in \bigcup\limits_{m<M} (t_m\cap {}^n 2):
  \rho \rest \big(n_\xi,n_\xi+N_0\big]=\nu^*_d\}=\emptyset$ if and only if 
    $d=\Theta(b,c,i,1)$ for some $b<c<\ell$ and $i<\iota$. 
\end{enumerate}
Now consider $\rho_i^0\rest \big(n_\xi,n_\xi+N_0\big]$,
$\rho_i^1\rest \big(n_\xi+1,n_\xi+N_0\big]$ for $i<\iota$. 
\medskip

If for some $(i,x)\neq (j,y)$ we have $\rho_i^x\rest
\big(n_\xi, n_\xi+N_0\big] =\rho_j^y\rest \big(n_\xi,n_\xi+N_0\big]$, then
(using $(\boxtimes)_a$--$(\boxtimes)_d$ and the linear independence of
$\nu_d$'s) we must also have that 
\[\rho_i^0\rest \big(n_\xi,n_\xi+N_0\big] =\rho_i^1\rest
\big(n_\xi,n_\xi+N_0 \big]\quad\mbox{  for all }i<\iota.\]
Thus, for every $i<\iota$ there are $b<c<\ell$ and $j<\iota$ such that 
\[\{\rho^0_i,\rho^1_i\}=\{g_j(\gamma_b,\gamma_c), g_j(\gamma_c,
\gamma_b)\}.\] 
Since for $b<c<\ell$ we have 
\[\big(g_j(\gamma_b,\gamma_c)+ g_j(\gamma_c,\gamma_b)\big)\rest
(N_0,N_0+\ell] =\langle \underbrace{0,\ldots, 0}_b\rangle\conc \langle
\underbrace{1,\ldots, 1}_{c-b}\rangle\conc \langle \underbrace{0,\ldots,
  0}_{\ell-c} \rangle\]
we immediately get that (in the current situation) for some $b<c<\ell$
we have    
\[\big\{\{\rho^0_i,\rho^1_i\}:i<\iota\big\}=\big\{\{
g_i(\gamma_b,\gamma_c),g_i(\gamma_c,\gamma_b)\}: i<\iota\big\}.\] 
\medskip

So let us assume that $\rho_i^x\rest \big(n_\xi, n_\xi+N_0\big] \neq
\rho_j^y\rest \big(n_\xi,n_\xi+N_0\big]$ for all distinct 
$(i,x),(j,y)\in \iota\times 2$. Since $\{{\bf 1},\nu_0,\ldots,
\nu_{N_0-2}\}$ are linearly independent we may use Lemma  \ref{litlem}(2) to
conclude that  
\[\Big\{\big\{\rho_i^0\rest \big(n_\xi,n_\xi+N_0\big], \rho_i^1\rest
\big(n_\xi, n_\xi+N_0\big]\big\}: i<\iota\Big\}\subseteq \Big\{\big\{\nu_d,
\nu^*_d\big\}: d<N_0-1\Big\}.\] 
Consequently, we easily deduce that 
\[\big\{\{\rho^0_i,\rho^1_i\}:i<\iota\big\}\subseteq \big\{\{g_i(\vare_0,
\vare_1), g_i(\vare_1,\vare_0)\}: i<\iota\ \&\ (\vare_0,\vare_1)\in
B\cup D\cup E \big\}.\]
Using the linear independence of $\eta^\xi_\vare$'s and the definitions of
$g_i$'s in $(*)_{18}$ one checks that the three sets 

$\{g_i(\vare_0,\vare_1)+g_i(\vare_1,\vare_0):(\vare_0,\vare_1)\in B,\
i<\iota\}$, 

$\{g_i(\vare_0,\vare_1)+g_i(\vare_1,\vare_0):(\vare_0,\vare_1)\in D,\
i<\iota\}$, 

$\{g_i(\vare_0,\vare_1)+g_i(\vare_1,\vare_0):(\vare_0,
\vare_1)\in E,\ i<\iota\}$

\noindent are pairwise disjoint. Therefore,
$\big\{\{\rho^0_i,\rho^1_i\}:i<\iota\big\}$ must be included in (exactly)
one of the sets 

$\big\{\{g_i(\vare_0, \vare_1),  g_i(\vare_1,\vare_0)\}: i<\iota\ \&\
(\vare_0, \vare_1)\in B\big\}$,

$\big\{\{g_i(\vare_0, \vare_1),  g_i(\vare_1,\vare_0)\}: i<\iota\ \&\
(\vare_0, \vare_1)\in D\big\}$, or 

$\big\{\{g_i(\vare_0, \vare_1),  g_i(\vare_1,\vare_0)\}: i<\iota\ \&\
(\vare_0, \vare_1)\in E\big\}$.

\noindent But now we easily check that for some $(\vare_0,\vare_1)\in 
B\cup D\cup E$ we must have  
\[\big\{\{\rho^0_i,\rho^1_i\}:i<\iota\big\}=\big\{\{g_i(\vare_0,
\vare_1), g_i(\vare_1,\vare_0)\}: i<\iota\big\}.\]

This completes the verification that $q=\big(w,n,M,\bar{\eta},
\bar{t},\bar{r}, \bar{h},\bar{g},\cM\big)\in \bbP$. 
It should be clear that $q$ is stronger than both $p_\xi$ and $p_\varsigma$.  
\end{proof}

Define $\bbP$--names $\name{T}_m$ and $\name{\eta}_\alpha$ (for
$m<\omega$ and $\alpha<\lambda$) by

$\forces_\bbP$`` $\name{T}_m= \bigcup\{t^p_m: p\in
  \name{G}_\bbP\ \wedge\ m<M^p\}$ '', and 

$\forces_\bbP$`` $\name{\eta}_\alpha= \bigcup\{\eta^p_\alpha:
  p\in \name{G}_\bbP\ \wedge\ \alpha\in w^p\}$ ''.

\begin{claim}
  \label{cl10}
  \begin{enumerate}
\item For each $m<\omega$ and $\alpha<\lambda$,

$\forces_\bbP$`` $\name{\eta}_\alpha\in\can$ and
$\name{T}_m\subseteq {}^{\omega>}2$ is a tree without terminal nodes ''. 
\item $\forces_{\bbP}$`` $\bigcup\limits_{m<\omega}
  \lim(\name{T}_m)$ is  a $2\iota$--{\bf npots} set ''.  
  \end{enumerate}
\end{claim}

\begin{proof}[Proof of the Claim]
(1) \quad By Claim \ref{cl8} (and the definition of the order in $\bbP$). 
\medskip

\noindent (2)\quad Let $G\subseteq \bbP$ be a generic filter over $\bV$ and 
let us work in $\bV[G]$. 

Let $k=2\iota$ and $\bar{T}=\langle (\name{T}_m)^G:m<\omega\rangle$. 

Suppose towards contradiction that $B=\bigcup\limits_{m<\omega}
\lim\big((\name{T}_m)^G\big)$ is  a $k$--{\bf pots} set. Then, by
Proposition \ref{eqnd},  $\NDRK(\bar{T})=\infty$. Using Lemma
\ref{lemonrk}(5,), by induction on $j<\omega$  we choose $\bm_j, \bm_j^*\in
\Mtk$ and $p_j\in G$ such that   
\begin{enumerate}
\item[(i)] $\ndrk(\bm_j)\geq\omega_1$, $|u_{\bm_j}|>5$ and $\bm_j\sqsubseteq 
  \bm^*_j \sqsubseteq \bm_{j+1}$,  
\item[(ii)] for each $\nu\in u_{\bm^*_j}$ the set $\{\eta\in u_{\bm_{j+1}}: 
  \nu\vtl \eta\}$ has at least two elements, 
\item[(iii)] $p_j\leq p_{j+1}$, $\ell_{\bm_j}\leq\ell_{\bm^*_j}= n^{p_j}<
  \ell_{\bm_{j+1}}$ and  $\rng(h_i^{\bm_j})\subseteq M^{p_j}$ for all
  $i<\iota$,  and   
\item[(iv)] $|\{\eta\rest n^{p_j}:\eta\in u_{\bm_{j+1}}\}|=|u_{\bm_j}|=|
  u_{\bm^*_j}|$.  
\end{enumerate}
Then, by (iii)+(iv), $\bm_j,\bm^*_j\in {\mathbf
  M}^{n^{p_j}}_{\bar{t}^{p_j},k}$. It follows from Claim \ref{cl7} that for
some $w_j\subseteq w^{p_j}$ and $\rho_j\in {}^{n^{p_j}}2$ we have
$(\bm_j^*+\rho_j) \doteqdot \bm^{p_j} (n^{p_j}, w_j)\in \cM^{p_j}$.  

Fix $j$ for a moment and consider $\bm^{p_j} (n^{p_j}, w_j)\in
\cM^{p_j}$ and $\bm^{p_{j+1}} (n^{p_{j+1}}, w_{j+1})\in \cM^{p_{j+1}}$.
Since $(\bm_j^*+(\rho_{j+1}\rest n^{p_j}))\sqsubseteq
(\bm^*_{j+1}+\rho_{j+1}) \doteqdot \bm^{p_{j+1}} (n^{p_{j+1}}, w_{j+1})$,
we may choose $w^*_j\subseteq w_{j+1}$ such that   
\[(\bm_j^*+(\rho_{j+1}\rest n^{p_j}))\doteqdot \bm^{p_{j+1}} (n^{p_j},
w^*_j)\sqsubseteq^*\bm^{p_{j+1}}(n^{p_{j+1}},w_{j+1})\]
 (and the latter two belong to $\cM^{p_{j+1}}$).  Then also 
\[\bm^{p_{j+1}} (n^{p_j},w^*_j)\doteqdot \bm^{p_j} (n^{p_j}, w_j)+
  (\rho_j+\rho_{j+1}\rest n^{p_j})= \bm^{p_{j+1}}(n^{p_j}, w_j)+
  (\rho_j+\rho_{j+1}\rest n^{p_j}),\]  
so by clause $(*)_9$ for $p_{j+1}$ we have 
\[\rk(w^*_j)=\rk(w_j).\] 
Clause (ii) of the choice of $\bm_{j+1}$ implies that 
\[(\forall \gamma\in w^*_j)(\exists\delta\in w_{j+1}\setminus
  w_j^*)( \eta^{p_{j+1}}_\gamma\rest n ^{p_j}= \eta_\delta^{p_{j+1}}
  \rest n^{p_j}).\] 
Let $\delta(\gamma)$ be the smallest $\delta\in w_{j+1}\setminus w_j^*$
with the above property and let $w^*_j(\gamma)=(w^*_j\setminus
\{\gamma\})\cup \{\delta(\gamma)\}$. Then, for $\gamma\in w^*_j$,
$\bm^{p_{j+1}}(n^{p_j}, w^*_j(\gamma))\in \cM^{p_{j+1}}$ and 
\[\bm^{p_{j+1}}(n^{p_j},  w^*_j)\doteqdot
  \bm^{p_{j+1}}(n^{p_j},w^*_j(\gamma))\sqsubseteq^*
  \bm^{p_{j+1}}(n^{p_{j+1}}, w_{j+1}).\] 
So by clause $(*)_9$ we know that for each $\gamma\in w_j$:
\[\rk(w^*_j(\gamma))=\rk(w^*_j),\quad \zeta(w^*_j(\gamma))
  =\zeta(w_j^*), \quad\mbox{ and }\quad k(w^*_j(\gamma))
  =k(w_j^*).\]
Let $n=|w^*_j|$, $\zeta=\zeta(w_j^*)$, $k=k(w_j^*)$, and let
$w_j^*=\{\alpha_0,\ldots,\alpha_k,\ldots, \alpha_{n-1}\}$ be the increasing
enumeration. Let $\alpha^*_k=\delta(\alpha_k)$. Then clause $(*)_9$ also
gives that $w_j^*(\alpha_k)=\{\alpha_0,\ldots,\alpha_{k-1},\alpha_k^*,
\alpha_{k+1}, \ldots, \alpha_{n-1}\}$ is the increasing enumeration.  Now, 
\[\begin{array}{l}
\bbM\models R_{n,\zeta}[\alpha_0,\ldots,\alpha_{k-1},\alpha_k,
\alpha_{k+1}, \ldots, \alpha_{n-1}]\qquad \mbox{ and}\\
\bbM\models R_{n,\zeta}[\alpha_0,\ldots,\alpha_{k-1},\alpha_k^*,
\alpha_{k+1}, \ldots, \alpha_{n-1}],
\end{array}\]
and consequently if $\rk(w^*_j)\geq 0$, then 
\[\rk(w_{j+1})\leq \rk(w^*_j\cup\{\alpha^*_k\})<\rk(w^*_j)=\rk(w_j)\] 
(remember $(\circledast)_{\rm d}$). 
\medskip

Now, unfixing $j$, suppose that we constructed $w_{j+1},w^*_j$ for all
$j<\omega$. It follows from our considerations above that for some
$j_0<\omega$ we must have:
\begin{enumerate}
\item[(a)] $\rk(w^*_{j_0})=-1$, and  
\item[(b)] $\bm^{p_{j_0+1}}(n^{p_{j_0}}, w^*_{j_0})\sqsubseteq^*
  \bm^{p_{j_0+1}}(n^{p_{j_0+1}}, w_{j_0+1})$ (and both belong to
  $\cM^{p_{j_0+1}}$), 
\item[(c)] for every $\alpha\in w^*_{j_0}$ we have 
\[\big|\big\{\beta\in w_{j_0+1}:\eta^{p_{j_0+1}}_\alpha\rest n^{p_{j_0}}
\vtl \eta^{p_{j_0+1}}_\beta\big\}\big| >1.\]  
\end{enumerate}
However, this contradicts clause $(*)_{10}$ (for $p_{j_0+1}$). 
\end{proof}
\end{proof}

\begin{corollary}
\label{MAgives}
Assume ${\bf MA}+\neg{\bf CH}$, $\lambda<\con$ and ${\rm
  NPr}_{\omega_1}(\lambda)$, $3\leq\iota<\omega$. Then there exists a
$\Sigma^0_2$ $2\iota$--{\bf npots}--set $B\subseteq\can$ which has $\lambda$
many pairwise  $2\iota$--nondisjoint translations. 
\end{corollary}

\begin{proof}
Standard modification of the proof of Theorem \ref{522fortranslate}.   
\end{proof}

\begin{corollary}
\label{corforotps}
Assume ${\rm NPr}_{\omega_1}(\lambda)$ and $\lambda=\lambda^{\aleph_0}
<\mu= \mu^{\aleph_0}$, $3\leq\iota<\omega$. Then there is a ccc forcing
notion $\bbQ$ of size $\mu$ forcing that  
\begin{enumerate}
\item[(a)] $2^{\aleph_0}=\mu$ and 
\item[(b)] there is a $\Sigma^0_2$ $2\iota$--{\bf npots}--set
  $B\subseteq\can$ which has $\lambda$ many pairwise $2\iota$--nondisjoint
  translates but not $\lambda^+$ such translates. 
\end{enumerate}
\end{corollary}

\begin{proof}
Let $\bbP$ be the forcing notion given by Theorem \ref{522fortranslate} and
let $\bbQ=\bbP*\bbC_\mu$. Use Proposition \ref{proptostart}(4) to argue that
the set $B$ added by $\bbP$ is a {\bf npots}--set in $\bV^\bbQ$. By
\ref{proptostart}(3) this set  cannot have $\lambda^+$ pairwise
$2\iota$--nondisjoint translates, but it does have $\lambda$ many pairwise
$2\iota$--nondisjoint translates (by absoluteness). 
\end{proof}

\begin{remark}
It follows from Proposition \ref{proptostart}(1,2), that if there exists a
$\Sigma^0_2$ {\bf pots}--set $B\subseteq \can$ such that for some set
$A\subseteq \can$ we have $(B+a)\cap (B+b)\neq \emptyset$ for all $a,b\in
A$, then $\stnd(B)\subseteq\can\times\can$ is a $\Sigma^0_2$ set which
contains a $|A|$--square but no perfect   square. Thus Corollary
\ref{corforotps} is a slight generalization of Shelah \cite[Theorem 
1.13]{Sh:522}. 
\end{remark}

\section{Open questions}

\begin{problem}
Is it consistent that for every Borel set $B\subseteq \can$, 
\begin{enumerate}
\item[if] there is $H\subseteq \can$ of size $\aleph_1$ such that $|(B+x)\cap 
  (B+y)|\geq 6$ for all $x,y\in H$, 
\item[then] there is a perfect set $P$ such that $|(B+x)\cap (B+y)|\geq 
  6$ for all $x,y\in P$ ? 
\end{enumerate}   
(Compare this with Proposition \ref{proptostart}(3).)
\end{problem}

\begin{problem}
Is is consistent to have a Borel set $B\subseteq \can$ such that 
\begin{itemize}
\item for some uncountable set $H$, $(B+x)\cap (B+y)$ is uncountable for
  every $x,y\in H$, but 
\item for every perfect set $P$ there are $x,y\in P$ with $(B+x)\cap (B+y)$
  countable?
\end{itemize}
\end{problem}

\begin{problem}
Is it consistent to have a Borel set $B\subseteq \can$ such that 
\begin{itemize}
\item $B$ has uncountably many pairwise disjoint translations, but
\item there is no perfect of pairwise disjoint translations of $B$. ? 
\end{itemize}
\end{problem}


\begin{thebibliography}{1}

\bibitem{BRSh:512}
Marek Balcerzak, Andrzej Roslanowski, and Saharon Shelah.
\newblock {Ideals without ccc}.
\newblock {\em {Journal of Symbolic Logic}}, 63:128--147, 1998.
\newblock arxiv:math.LO/9610219.

\bibitem{BaJu95}
Tomek Bartoszy\'nski and Haim Judah.
\newblock {\em {Set Theory: On the Structure of the Real Line}}.
\newblock A K Peters, Wellesley, Massachusetts, 1995.

\bibitem{J}
Thomas Jech.
\newblock {\em {Set theory}}.
\newblock Springer Monographs in Mathematics. Springer-Verlag, Berlin, 2003.
\newblock The third millennium edition, revised and expanded.

\bibitem{RoRyxx}
Andrzej Ros{\l}anowski and Vyacheslav~V. Rykov.
\newblock {Not so many non-disjoint translations},
\newblock submitted.
\newblock arxiv:1711.04058.

\bibitem{Sh:F1771}
Andrzej Ros{\l}anowski and Saharon Shelah.
\newblock {Borel sets without perfectly many overlapping translations,
II.} 
\newblock In progress.
\newblock 

\bibitem{Sh:522}
Saharon Shelah.
\newblock {Borel sets with large squares}.
\newblock {\em {Fundamenta Mathematicae}}, 159:1--50, 1999.
\newblock arxiv:math.LO/9802134.

\bibitem{Zak13}
Piotr Zakrzewski.
\newblock {On Borel sets belonging to every invariant ccc $\sigma$--ideal on
  $2^{\mathbb N}$}.
\newblock {\em Proc. Amer. Math. Soc.}, 141:1055--1065, 2013.

\end{thebibliography}

\end{document}